\documentclass[12pt,letterpaper]{article}

\usepackage{amssymb,amsfonts,amscd,amsthm}
\usepackage{thmtools}
\usepackage[all,arc]{xy}
\usepackage{bbm}
\usepackage{mathrsfs}
\usepackage{tikz}
\usepackage[left=0.9in,top=0.9in,right=0.9in,bottom=0.9in]{geometry}
\usepackage{mathtools}
\usepackage{hyperref}

\usepackage{color}
\usepackage{graphicx}
\usepackage{enumerate}
\graphicspath{ {TexImages/} }
\usepackage{enumitem}
\usepackage{subcaption}

\hypersetup{
	colorlinks,
	citecolor=blue,
	filecolor=blue,
	linkcolor=blue,
	urlcolor=blue,
	linktocpage
}
\usepackage{cleveref}

\newtheorem{theorem}{Theorem}[section]
\newtheorem{corollary}[theorem]{Corollary}
\newtheorem{proposition}[theorem]{Proposition}
\newtheorem{lemma}[theorem]{Lemma}
\newtheorem{claim}[theorem]{Claim}

\newtheorem{problem}[theorem]{Problem}

\newtheorem{conjecture}[theorem]{Conjecture}
\newtheorem{observation}[theorem]{Observation}

\theoremstyle{definition}
\newtheorem{definition}{Definition}

\title{}
\author{}
\date{}

\newcommand{\pow}{q} 
\newcommand{\Pow}{{q_\star}} 
\newcommand{\ex}{\mathrm{ex}}
\renewcommand{\c}[1]{\mathcal{#1}}
\newcommand{\Om}{\Omega}

\newcommand{\sm}{\setminus}
\newcommand{\sub}{\subseteq}
\newcommand{\E}{\mathbb{E}}
\newcommand{\mon}{\mathrm{mon}}
\newcommand{\Mon}{\mathrm{Mon}}
\newcommand{\ep}{\epsilon}
\newcommand{\half}{\frac{1}{2}}

\title{Rational Exponents for General Graphs}

\author{Sean English\thanks{University of North Carolina Wilmington, \texttt{EnglishS@uncw.edu}.} \and Sam Spiro\thanks{Georgia State University, \texttt{sspiro@gsu.edu}.}}

\title{Helly Theorems for Generalized Tur\'an Problems}
\begin{document}
	\maketitle 
	
	\begin{abstract}
		Given a graph $T$ and a family of graphs $\mathcal{F}$, the generalized Tur\'an number $\mathrm{ex}(n,T,\mathcal{F})$ is the maximum number of copies of $T$ in an $n$-vertex $\mathcal{F}$-free graph.  We prove a general theorem which states that for any tree $T$, any family $\mathcal{F}$, and any integer $k$, either $\mathrm{ex}(n,T,\mathcal{F})$ is at least $\Omega(n^{k+1})$ or at most $O(\mathrm{ex}(n,\mathcal{F})^{k})$, from which we derive a number of consequences. Our proofs rely on  new variants of the classical Helly Theorem for trees which may be of independent interest. As far as we are aware, this is the first known application of Helly theorems for Tur\'an type problems.
	\end{abstract}
	
	\section{Introduction}
	\subsection{Generalized \texorpdfstring{Tur\'an}{Turan} Problems}
	
	One of the central areas of  extremal combinatorics is the study of Tur\'an numbers $\ex(n,\c{F})$ for families of graphs $\c{F}$, which is defined to be the maximum number of edges in an $n$-vertex $\c{F}$-free graph.  More generally, given a graph $H$ and a family of graphs $\c{F}$, we define the generalized Tur\'an number $\ex(n,H,\c{F})$ to be the maximum number of copies of $H$ in an $n$-vertex $\c{F}$-free graph.  Note that $\ex(n,K_2,\c{F})=\ex(n,\c{F})$, so this is a generalization of the classical Tur\'an number.
	
	The systematic study of generalized Tur\'an numbers was initiated by Alon and Shikhelman~\cite{AS2016}, and since then there has been a large number of results dedicated to the topic,
	see e.g. \cite{GGMV2020, GP2022, HP2021, MYZ2018, ZGHLSX2023}.  For a more thorough look at generalized Tur\'an numbers we refer the
	interested reader to the excellent recent survey by  Gerbner and Palmer~\cite{GP2025}.
	
	In this paper we focus on generalized Tur\'an numbers when $H$ is a tree.  Despite this being a natural extension of the $H=K_2$ case of the classical Tur\'an problem, there are surprisingly few results in this direction in the literature.  Perhaps the first result of this form is the following which appeared in the original paper of Alon and Shikehlman.
	
	\begin{theorem}[\cite{AS2016}]\label{Alon Shikelman trees}
		If $T,F$ are trees and if $\ex(n,T,F)>0$ for $n$ sufficiently large, then $\ex(n,T,F)=\Theta(n^k)$ for some integer $k$.
	\end{theorem}
	While \Cref{Alon Shikelman trees} is technically a result about generalized Tur\'an numbers of trees, it also ``misses the forest for the trees'' in the sense that Letzter~\cite{L2019} showed that \Cref{Alon Shikelman trees} continues to hold with $T$ an arbitrary graph.  Outside of \Cref{Alon Shikelman trees}, the only other results which consider generalized Tur\'an numbers for arbitrary trees $T$ is a paper by  Gerbner~\cite{G2023} who determined the order of magnitude of $\ex(n,T,K_{2,t})$ for all trees $T$, and an observation of Cambie, de Joannis de Verclos, and Kang~\cite{CJK23} which determines the exact value of $\ex(n,T,K_{1,t})$ for all $T$.  In terms of results for specific choices of trees $T$, the most substantial result we know of is a nice theorem of F\" uredi and K\" undgen~\cite{FK2006} which gives essentially optimal bounds on $\ex(n,K_{1,t},\c{F})$ based on the classical Tur\'an number $\ex(n,\c{F})$.  A slightly weakened version of this result can be stated as follows. 
	
	\begin{theorem}[\cite{FK2006}]\label{FurediKundgen}
		Let $\ell$ be an integer and $\c{F}$ a family of graphs which does not contain a subgraph of a star.  If $\ex(n,\c{F})=\Theta(n^{2-\beta})$ for some $\beta<\ell^{-1}$, then
		\[\ex(n,K_{1,\ell},\c{F})=\Theta(\ex(n,\c{F})^\ell n^{1-\ell}),\]
		and if $\beta\ge \ell^{-1}$ then
		\[\ex(n,K_{1,\ell},\c{F})=\Omega(n^\ell).\]
	\end{theorem}
	
	The main result of this paper is a general bound on $\ex(n,T,\c{F})$ for trees $T$ which holds for any choice of $T$ and any family $\c{F}$, thereby greatly expanding on the limited set of previously known results listed above.  
	
	Our result is spiritually similar to that of \Cref{FurediKundgen} in the sense that it either bounds $\ex(n,T,\c{F})$ as a function of $\ex(n,\c{F})$ or lower bounds it by some function independent of $\ex(n,\c{F})$.  Specifically, our theorem implies that for any tree $T$, family $\c{F}$, and integer $k$, either $\ex(n,T,\c{F})=\Omega(n^{k+1})$ or $\ex(n,T,\c{F})=O(\ex(n,\c{F})^{k})$.  Our theorem moreover gives a criterion for determining whether the bound $\Omega(n^{k+1})$ or $O(\ex(n,\c{F})^{k})$ holds for a given family $\c{F}$, and to state this condition we need a technical definition.
	
	\begin{definition}\label{definition flower powers}
		Given a graph $H$, a set of vertices $R\sub V(H)$, and an integer $\pow$, we define the graph $H_R^\pow$ to be the graph obtained by taking the union of $\pow$ copies of $H$ such that each of the copies agree on the set of vertices $R$ and are otherwise disjoint.  For each integer $k\ge 1$ we define
		\[
		\c{F}_{H,k}^\pow:=\{H_R^\pow: R\subsetneq V(H),\ H-R\textrm{ has at least }k\textrm{ connected components}\}.
		\]
	\end{definition}
	
	See Figure~\ref{figure flower power definition full figure} for an example of a graph of the form $H_R^\pow$.  With this our main theorem can be stated as follows.
	
	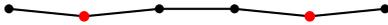
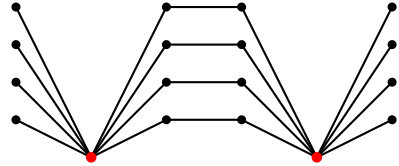
\begin{figure}
		\begin{center}
			\begin{subfigure}[t]{0.4\textwidth}
				\begin{center}\begin{tikzpicture}
						\draw[thick] (0,0.1)--(1,0)--(2,0.1)--(3,0.1)--(4,0)--(5,0.1);
						
						\node (0) at (0,0.1) {};
						\node (1) at (1,0) {};
						\node (2) at (2,0.1) {};
						\node (3) at (3,0.1) {};
						\node (4) at (4,0) {};
						\node (5) at (5,0.1) {};
						
						\fill (0) circle (0.06) node [above] {};
						\fill[color=red] (1) circle (0.07) node [above] {};
						\fill (2) circle (0.06) node [above] {};
						\fill (3) circle (0.06) node [above] {};
						\fill[color=red] (4) circle (0.07) node [above] {};
						\fill (5) circle (0.06) node [above] {};
					\end{tikzpicture}
				\end{center}
				\caption{$H=P_6$. The vertices in $R$ are red.}\label{figure flower power definition path}
			\end{subfigure}\hspace{0.1\textwidth}\begin{subfigure}[t]{0.4\textwidth}
				\begin{center}
					\begin{tikzpicture}
						\draw[thick] (0,0.5)--(1,0)--(2,0.5)--(3,0.5)--(4,0)--(5,0.5);
						\draw[thick] (0,1)--(1,0)--(2,1)--(3,1)--(4,0)--(5,1);
						\draw[thick] (0,1.5)--(1,0)--(2,1.5)--(3,1.5)--(4,0)--(5,1.5);
						\draw[thick] (0,2)--(1,0)--(2,2)--(3,2)--(4,0)--(5,2);
						
						\node (0) at (0,0.5) {};
						\node (1) at (1,0) {};
						\node (2) at (2,0.5) {};
						\node (3) at (3,0.5) {};
						\node (4) at (4,0) {};
						\node (5) at (5,0.5) {};
						
						\node (6) at (0,1.5) {};
						\node (7) at (2,1.5) {};
						\node (8) at (3,1.5) {};
						\node (9) at (5,1.5) {};
						
						\node (10) at (0,1) {};
						\node (11) at (2,1) {};
						\node (12) at (3,1) {};
						\node (13) at (5,1) {};
						
						\node (14) at (0,2) {};
						\node (15) at (2,2) {};
						\node (16) at (3,2) {};
						\node (17) at (5,2) {};
						
						\fill (0) circle (0.06) node [above] {};
						\fill[color=red] (1) circle (0.07) node [above] {};
						\fill (2) circle (0.06) node [above] {};
						\fill (3) circle (0.06) node [above] {};
						\fill[color=red] (4) circle (0.07) node [above] {};
						\fill (5) circle (0.06) node [above] {};
						\fill (6) circle (0.06) node [above] {};
						\fill (7) circle (0.06) node [above] {};
						\fill (8) circle (0.06) node [above] {};
						\fill (9) circle (0.06) node [above] {};
						\fill (10) circle (0.06) node [above] {};
						\fill (11) circle (0.06) node [above] {};
						\fill (12) circle (0.06) node [above] {};
						\fill (13) circle (0.06) node [above] {};
						\fill (14) circle (0.06) node [above] {};
						\fill (15) circle (0.06) node [above] {};
						\fill (16) circle (0.06) node [above] {};
						\fill (17) circle (0.06) node [above] {};
					\end{tikzpicture}
				\end{center}
				\caption{The graph $H_R^4$.}\label{figure flower power definition H_R^4}
			\end{subfigure}
		\end{center}
		\caption{An example of the graph $H_R^4$ from \Cref{definition flower powers} when $H$ is the path on $6$ vertices and $R$ consists of the set of red vertices in Figure~\ref{figure flower power definition path}. $H_R^4$ is a member of the family $\mathcal{F}_{H,3}^4$ since $H-R$ has (at least) three connected components.}\label{figure flower power definition full figure}
	\end{figure}

	\begin{theorem}\label{k-1 edges for trees}
		Let $T$ be a tree, $\c{F}$ a family of graphs, and $k\ge 0$ an integer.  If there exists an integer $q$ such that every $\c{F}$-free graph is $\c{F}_{T,k+1}^q$-free, then
		\[
		\ex(n,T,\c{F})=O_{T,k,q}(\ex(n,\c{F})^{k}),
		\]
		and if no such $q$ exists then
		\[
		\ex(n,T,\c{F})=\Om_{T,k}(n^{k+1}).
		\]
	\end{theorem}
	
	One can derive a number of powerful results using \Cref{k-1 edges for trees} together with some very short proofs.  For example, one can prove the general statement that families $\c{F}$ with $\ex(n,\c{F})$ small have $\ex(n,T,\c{F})$ close to an integer whenever $T$ is a tree.
	
	\begin{theorem}\label{tree generalized exponents close to integers}
		Let $T$ be a tree and $\c{F}$ a family with $\ex(n,\c{F})=O(n^{1+\ep})$.  If $\ex(n,T,\c{F})>0$ for all sufficiently large $n$, then there exists an integer $k$ such that
		\[
		\ex(n,T,\c{F})=\Omega(n^k),
		\] 
		and
		\[
		\ex(n,T,\c{F})=O(n^{k+k\ep}).
		\]
	\end{theorem}

	In particular, the case $\ep=0$ gives the following.
	\begin{corollary}\label{tree containing forest gives integers}
		For every tree $T$ and family $\c{F}$ containing a forest, if $\ex(n,T,\c{F})>0$ for all sufficiently large $n$, then there exists an integer $k$ such that $\ex(n,T,\c{F})=\Theta(n^k)$.
	\end{corollary}
	
	This gives a two-way generalization of \Cref{Alon Shikelman trees}, in the sense that we can forbid forests and not just trees, and we can also forbid an entire family of graphs rather than just a single graph.  We emphasize that this result can be derived from the original method of \Cref{Alon Shikelman trees}; we highlight its particular statement here both to show the power of what can be proven using \Cref{k-1 edges for trees}, and also because the exact statement of \Cref{tree containing forest gives integers} can be used to prove a sort of ``stability'' result for generalized Tur\'an problems for trees.
	\begin{corollary}\label{leaf stability corollary}
		If $T\ne K_2$ is a tree with $\ell\ge 2$ leaves and if $\c{F}$ is a family of graphs with $\ex(n,T,\c{F})=O(n^\ell)$ and $\ex(n,T,\c{F})>0$ for all sufficiently large $n$, then $\ex(n,T,\c{F})=\Theta(n^k)$ for some integer $k$.
	\end{corollary}
	In particular, if $\ex(n,T,\c{F})=o(n^k)$ for any integer $k\le \ell$, then the generalized Tur\'an number must in fact drop all the way down to $O(n^{k-1})$.  This sort of stability result is strongest when $T$ has many leaves. It is well known that trees with few leaves have large diameter, and for such trees the following (more involved) consequence of \Cref{k-1 edges for trees} gives effective stability.
	
	\begin{theorem}\label{diameter stability}
		Let $k\ge 1$ be an integer. If $T$ is a tree of diameter $d\ge 2k$, then for every family of graphs $\c{F}$ with $\ex(n,T,\c{F})=o(n^{k+1})$, we have
		\[
		\ex(n,T,\c{F})=O(n^{k+\frac{k(k-1)}{d-k}}).
		\]
	\end{theorem}
	
	As we discuss below, this upper bound is tight for paths $P_{d+1}$ with $d\equiv 1 \mod k-1$.  Combining the argument used to prove \Cref{diameter stability} together with \Cref{leaf stability corollary} allows us to show that all ``small'' generalized Tur\'an numbers for trees must be very close to an integer power of $n$.

	\begin{theorem}\label{diameter corollary}
		For every tree $T$ and $0<\ep\le 1$, if $\c{F}$ is a family of graphs such that \[\ex(n,T,\c{F})=O(n^{(\ep v(T))^{1/3}}),\]  and such that $\ex(n,T,\c{F})>0$ for all sufficiently large $n$; then there exists some integer $k$ such that \[\ex(n,T,\c{F})=\Om(n^k),\]
		and
		\[\ex(n,T,\c{F})=O(n^{k+\ep}).\]
	\end{theorem}
	This in turn yields the surprising fact that every non-integral exponents can be realized by only finitely many trees.
	
	\begin{corollary}\label{non-integral exponents}
		For every non-integral number $r\in \mathbb{R}_{\ge 0}\sm \mathbb{Z}$, there exists at most finitely many trees $T$ with the property that there exists a family of graphs $\c{F}$ satisfying $\ex(n,T,\c{F})=\Theta(n^r)$.
	\end{corollary}
	
	One can use \Cref{k-1 edges for trees} to show that for every tree $T$ there exists a (finite) family $\c{F}$ with $\ex(n,T,\c{F})=\Theta(n^\ell)$ where $\ell$ is the number of leaves of $T$, and this implies that the assumption $r\notin \mathbb{Z}$ is needed for \Cref{non-integral exponents} to hold.
	
	The applications above show the utility of our main result \Cref{k-1 edges for trees}. However, it is natural to ask whether these bounds are tight.  In Appendix~\ref{section appendix} we show that the upper bound of \Cref{k-1 edges for trees} (as well as that of \Cref{diameter stability}) is tight whenever $T=P_t$ is the path graph on $t$ vertices and $\c{F}=\c{F}_{P_t,k+1}^q$ with $q$ large provided $t\equiv 2\mod k-1$.  In particular, this shows that for every $k$ there exist infinitely many trees for which the upper bound of \Cref{k-1 edges for trees} is best possible.
	
	At this point the reader might ask what happens for paths with $t\not \equiv 2\mod k-1$.  In this case it turns out we can refine \Cref{k-1 edges for trees} to give a further sharpening of the upper bound.
	
	\begin{theorem}\label{k-2 edges for paths}
		Let $\c{F}$ be a family of graphs, and $k,t\ge 3$ integers with $t\not \equiv 2\mod k-1$.  If there exists an integer $q$ such that every $\c{F}$-free graph is $\c{F}_{P_t,k+1}^q$-free, then
		\[
		\ex(n,P_t,\c{F})=O_{t,k,q}(\ex(n,\c{F})^{k-1}\cdot n),
		\]
		and if no such $q$ exists then
		\[
		\ex(n,P_t,\c{F})=\Om_{t,k}(n^{k+1}).
		\]
	\end{theorem}
	That is, under the condition $t\not \equiv 2\mod k-1$ we can improve upon the upper bound of \Cref{k-1 edges for trees} by replacing one of the $\ex(n,\c{F})$ terms with $n$. This result turns out to be tight whenever $t\equiv 1\mod k-1$; see Appendix~\ref{section appendix} for details.  In addition to being of interest in its own right, we believe \Cref{k-2 edges for paths} serves as a proof-of-concept that further refinements can be made for other trees.  In particular, we comment on what might be true for arbitrary path lengths in the concluding remarks.

	
	\subsection{Helly Theorems}
	One of our main tools for proving \Cref{k-1 edges for trees} and its refinement \Cref{k-2 edges for paths} are generalizations of the classical Helly theorem for trees which we believe are of independent interest.
	
	Historically, the first Helly-type theorems appeared in convex geometry when Helly~\cite{H1923} proved that if $\c{C}$ is a collection of convex sets in $\mathbb{R}^d$ such that any subcollection $\c{C}'\sub \c{C}$ of size at least $d+1$ contains a point $v'$ in every element of $\c{C}'$, then there exists a point $v$ which lies in every element of $\c{C}$.  There have since been many other Helly-type theorems which have been studied within convex geometry and beyond \cite{H0,H1,H2,H3,H4}.
	
	In order to recall the statement of the classical Helly Theorem for trees, as well as to state our new variant of this result, we introduce some definitions.

	\begin{definition}
		Given a tree $T$, a \textbf{subtree system} $\mathcal{T}$ of $T$ is a collection of subtrees of $T$.  We say that a set of vertices $U\sub V(T)$ \textbf{pierces} a subtree $T'\sub T$ if $V(T')\cap U\ne \emptyset$, and we say that a set $U$ pierces a subtree system $\mathcal{T}$ if it pierces every $T'\in \c{T}$.  Similarly, we say that a set of edges $E\sub E(T)$ \textbf{pierces} a subtree $T'\sub T$ if a vertex in $V(T')$ is incident with an edge in $E$, and we say $E$ pierces a subtree system $\mathcal{T}$ if it pierces every $T'\in \c{T}$.
		
	\end{definition}
	
	In this language, the classical Helly Theorem for trees states that if $\c{T}$ is a subtree system such that any pair of elements from $\c{T}$ is pierced by a single vertex, then there exists a vertex which pierces all of $\mathcal{T}$. Note that the case of $T$ being a path essentially corresponds to the $d=1$ case of the classical Helly Theorem.
	
	To prove \Cref{k-1 edges for trees}, we will need a variant of the Helly Theorem for trees which guarantees a set of \textit{edges} which pierces $\c{T}$.
	
	\begin{theorem}[Edge Helly Theorem]\label{edge Helly}
		Let $T$ be a tree, $\c{T}$ a subtree system of $T$, and $k\ge 0$ an integer.  If every subcollection $\c{T}'\sub \c{T}$ of size at most $k+1$ can be pierced by at most $k$ edges, then $\c{T}$ can be pierced by at most $k$ edges.
	\end{theorem}
	
	As far as we are aware, our result \Cref{k-1 edges for trees} is the first known Tur\'an result which is proven using Helly-type theorems.  Moreover, one can show that our Edge Helly Theorem implies the classical Vertex Helly Theorem, and as such it is at least as strong as this classical result.
	
	In order to prove \Cref{k-2 edges for paths} for paths, one might expect that we need a Helly-type Theorem for paths which involves pirecing by a set of edges and a single vertex.  While our argument implicitly uses ideas of this form, ultimately the proof requires us to use a more technical notion of piercing which we call ``pseudo-piercing''. We defer the precise formulation of pseudo-piercing until \Cref{definition pseudo piercing} due to it's technical nature, but we emphasize here that we believe the notion of pseudo-piercing is novel and may provide avenues for future applications of Helly-type arguments to settings where such methods have not previously been used.
	
	\subsection{Organization, Notation, and Conventions}
	The rest of the paper is organized as follows. In \Cref{section applications}, we show how to use \Cref{k-1 edges for trees} to prove Theorems~\ref{tree generalized exponents close to integers},~\ref{diameter stability} and~\ref{diameter corollary}, along with Corollaries~\ref{tree containing forest gives integers},~\ref{leaf stability corollary} and~\ref{non-integral exponents}. In \Cref{section generalized Turan preliminaries}, we give some technical definitions and preliminary results, incuding the proof of \Cref{edge Helly}, which will all be necessary for proving our main theorem (\Cref{k-1 edges for trees}) and also \Cref{k-2 edges for paths}. In \Cref{section main theorem proof}, we prove \Cref{k-1 edges for trees} and also introduce the notion of pseudo-piercing. In \Cref{section refinement for paths}, we show how to refine \Cref{k-1 edges for trees} in the case of paths to prove \Cref{k-2 edges for paths}. Finally, in \Cref{section concluding remarks}, we provide some concluding remarks and open questions.

	We note that if $\mathcal{F}$ contains an empty graph, then $\ex(n,H,\mathcal{F})$ is not defined for large $n$.  As such we will silently assume through this work that every graph in $\mathcal{F}$ contains at least one edge.
	
	Given a graph $H$ we will write $v(H):=|V(H)|$ and $e(H):=|E(H)|$.  In our results and proofs involving a graph $H$ which we count within another graph $G$, we will often add hats to vertices and sets associated to $H$ in order to distinguish them from vertices and sets associated to $G$.  For example, $\hat{x}$ will refer to a vertex in $H$ while $x$ refers to one in $G$.
	
	For integers $a,b,c\ge 1$, we define the \textit{generalized theta graph} $\theta_{a,b,c}$ to be the graph obtained by taking $a$ paths on $1+cb$ vertices which all agree on the vertices $1,1+b,1+2b,\ldots,1+cb$ and which are otherwise pairwise disjoint.  Equivalently, $\theta_{a,b,c}=(P_{1+cb})_R^a$ where $R=\{v_1,v_{1+b},\ldots,v_{1+cb}\}$. When $c=1$ the graph $\theta_{a,b,c}$ is simply a so-called \emph{theta graph}, and we write $\theta_{a,b}:=\theta_{a,b,1}$.

	\section{Applications}\label{section applications}
	
	In this section we establish the applications of \Cref{k-1 edges for trees} from the introduction in the order they were stated.  Most of these applications have succinct proofs, with the one minor exception being the diameter result \Cref{diameter stability}.
	
	\begin{proof}[Proof of \Cref{tree generalized exponents close to integers}]
		Let $k\ge 0$ be the largest integer such that $\ex(n,T,\c{F})=\Omega(n^k)$.  Because $\ex(n,T,\c{F})=\Om(n^{k+1})$ does not hold, by \Cref{k-1 edges for trees} we must have
		\[
		\ex(n,T,\c{F})=O(\ex(n,\c{F})^{k})=O(n^{k+k\ep}).\qedhere
		\]
	\end{proof}
	\Cref{tree containing forest gives integers} follows immediately from \Cref{tree generalized exponents close to integers} and the fact that $\ex(n,\c{F})=O(n)$ whenever $\c{F}$ contains a forest (see e.g. Theorem 2.32 in~\cite{FS2013}), so we omit a formal proof.
	
	\begin{proof}[Proof of \Cref{leaf stability corollary}]
		If $\c{F}$ contains a forest then the result follows from \Cref{tree containing forest gives integers}.  Otherwise, consider the graph $G$ obtained by duplicating each leaf of $T$ a total of $n/v(T)$ times.  It is not difficult to see that $G$ is a tree for $T\ne K_2$, and therefore is $\c{F}$-free since the only subgraphs of $G$ are forests.  Moreover, $G$ has at most $n$ vertices and contains at least $\Omega(n^\ell)$ copies of $T$, implying that $\ex(n,T,\c{F})=\Omega(n^\ell)$.  This combined with our hypothesis gives $\ex(n,T,\c{F})=\Theta(n^\ell)$.
	\end{proof}
	
	Before we can prove \Cref{diameter stability}, we need a bound on the extremal number of the generalized theta graph $\theta_{a,b,c}$, which we will achieve by combining a classical result of Bondy and Simonvits~\cite{BS1974} with a recent result of Dong, Gao, and Liu~\cite{DGL2025}. To state this latter result, we need a technical definition:  given two graphs $H_1,H_2$ and vertices $u_1\in V(H_1)$ and $u_2\in V(H_2)$, define $H_1^{u_1}\odot H_2^{u_2}$ to be the graph obtained by taking the union of $H_1$ and $H_2$ and identifying $u_1$ with $u_2$.
	
	\begin{theorem}[\cite{DGL2025}]\label{gluing}
		If $H_1,H_2$ are two copies of the same connected bipartite graph $H$ and if $u_1\in V(H_1)$ and $u_2\in V(H_2)$ come from the same part of the bipartition, then 
		\[
		\ex(n,H_1^{u_1}\odot H_2^{u_2})=\Theta(\ex(n,H)).
		\]
	\end{theorem}
	
	\begin{corollary}\label{gluing theta graphs}
		We have
		\[
		\ex(n,\theta_{a,b,c})=O_{a,c}(n^{1+1/b}).
		\]
	\end{corollary}
	
	\begin{proof}
		We will prove this in the case when $c$ is a power of 2.  This will imply the result in general since for any $c$, if $r$ is such that $2^{r-1}<c\le 2^r$ then $\theta_{a,b,c}\sub \theta_{a,b,2^r}$ and hence 
		\[
		\ex(n,\theta_{a,b,c})\le \ex(n,\theta_{a,b,2^r})=O_{a,2^r}(n^{1+1/b})=O_{a,c}(n^{1+1/b}).
		\]
		
		The base case follows from the classical upper bound of $\ex(n,\theta_{a,b})=O_a(n^{1+1/b})$ due to Bondy and Simonvits~\cite{BS1974} for theta graphs.  Observe also that for $r\ge 1$ we have $\theta_{a,b,2^r}=\theta_{a,b,2^{r-1}}^u \odot \theta_{a,b,2^{r-1}}^u$ where $u$ is the first vertex of the paths.  It follows from \Cref{gluing} and induction that $\ex(n,\theta_{a,b,2^r})=O_{a,2^r}(n^{1+1/b})$, proving the result.
	\end{proof}
	
	The other result we need to prove \Cref{diameter stability} is a folklore result which allows us to ignore ``tree-like'' structures when computing the extremal number of graphs which are not forests. Recall that the $2$-core of a graph $F$, which we will denote as $c_2(F)$, is the subgraph of $F$ obtained by iteratively deleting vertices of degree at most $1$.
	
	\begin{lemma}\label{lemma 2core}
		If $F$ contains a cycle, then $\ex(n,F)=\Theta_F(\ex(n,c_2(F)))$.
	\end{lemma}
	
	Putting these pieces together gives a bound on the classical Tur\'an number for the family $\c{F}_{T,k+1}^q$ in terms of the diameter of $T$.
	
	\begin{proposition}\label{family bound given diameter}
		If $k\ge 1$ is an integer and if $T$ is a tree with diameter $d\ge 2k$, then
		\[\ex(n,\c{F}_{T,k+1}^q)=O_q(n^{1+\frac{k-1}{d-k}}).\]
	\end{proposition}
	\begin{proof}
		We first consider the case $k\ge 2$.  Let $P=v_0v_1\dots v_d$ be a longest path in $T$, set $b:=\left\lfloor \frac{d-2}{k-1}\right\rfloor\geq 2$, and let \[R:=\{v_{1+ib}\mid 0\leq i\leq k-2\}\}.\] Then $T-R$ has at least $k+1$ connected components since the vertices $v_0,v_b,\dots,v_{(k-1)b},v_d$ all lie in distinct components (with this implicitly using $d> 1+(k-1)b$), so $T_R^q\in \mathcal{F}_{T,k+1}^q$. Furthermore, it is straightforward to see that $c_2(T_R^q)=\theta_{q,b,k-1}$.  Putting this all together and applying \Cref{lemma 2core} and \Cref{gluing theta graphs}, we have
		\begin{equation}
			\ex(n,\mathcal{F}_{T,k+1}^q)\leq \ex(n,T_R^q)=O(\ex(n,\theta_{q,b,k-1}))=O_q(n^{1+1/b})=O_q(n^{1+\frac{k-1}{d-k}}),
		\end{equation}\label{eq:diameter bound}
		with this last inequality using $b\ge \frac{(d-2)-(k-2)}{k-1}=\frac{d-k}{k-1}$,
		proving the result.
		
		For the case $k=1$, we aim to show $\ex(n,\c{F}_{T,2}^q)=O_q(n)$ whenever $T$ has diameter at least 2, i.e.\ whenever $T\ne K_1,K_2$.  Let $R\sub V(T)$ denote the set of non-leaves of $T$.  Since $T\ne K_1,K_2$, the graph $T-R$ consists of isolated vertices for each of the at least 2 leaves of $T$, so $T_R^q\in \c{F}_{T,2}^q$.  Moreover, $T_R^q$ is simply the tree obtained by duplicating each leaf $q$ times, so
		\[\ex(n,\c{F}_{T,2}^q)\le \ex(n,T_R^q)=O_q(n),\]
		proving the result.
	\end{proof}
	This together with our main theorem quickly gives our results related to the diameter of $T$.
	\begin{proof}[Proof of \Cref{diameter stability}]
		Since $\ex(n,T,\c{F})=o(n^{k+1})$, \Cref{k-1 edges for trees} gives a $q$ such that every graph which is $\mathcal{F}$-free is also $\mathcal{F}_{T,k+1}^q$-free, and \Cref{k-1 edges for trees} also implies that  
		\[
		\ex(n,T,\c{F})=O(\ex(n,\c{F})^{k})=O(\ex(n,\c{F}_{T,k+1}^q)^k)=(n^{k+\frac{k(k-1)}{d-k}}),
		\]
		with this last step using \Cref{family bound given diameter}.
	\end{proof}
	
	Our next corollary relies on the fact that trees either have many leaves or large diameter.  This statement can be made precise through the following result, which is a slight weakening of a statement from~\cite{L1975}.
	
	\begin{lemma}[\cite{L1975}]\label{leaf diameter tradeoff}
		If $T$ is a tree with $\ell$ leaves and diameter $d$, then $v(T)\le \half \ell d$.
	\end{lemma}
	
	\begin{proof}[Proof of \Cref{diameter corollary}]
		Let $\ell$ and $d$ denote the number of leaves and the diameter of $T$.  If $\ex(n,T,\c{F})=O(n^\ell)$ then $\ex(n,T,\c{F})=\Theta(n^k)$ for some integer $k$ by \Cref{leaf stability corollary}, giving the result.  We may thus assume that this is not the case, and since $\ex(n,T,\c{F})=O(n^{(\ep v(T))^{1/3}})$, this with \Cref{leaf diameter tradeoff} implies $(\ep v(T))^{1/3}>\ell\ge 2v(T)/d$, and hence
		\begin{equation}\label{eq:d bound}
			d/2\ge v(T)/(\ep v(T))^{1/3}.
		\end{equation}
		Now let $k$ be the largest integer such that $\ex(n,T,\c{F})=\Omega(n^k)$. Since $\ex(n,T,\c{F})=O(n^{(\ep v(T))^{1/3}})$, we have
		\begin{equation}\label{equation d geq 2k}
			k\le (\ep v(T))^{1/3}\le v(T)/(\ep v(T))^{1/3}\le d/2,
		\end{equation}
		with this second inequality using the fact that $\ep\le 1$.
		
		Because we do not have $\ex(n,T,\c{F})=\Omega(n^{k+1})$, \Cref{k-1 edges for trees} implies there exists some $q$ such that every $\c{F}$-free graph is $\c{F}_{T,k+1}^q$-free and also that 
		\[
		\ex(n,T,\c{F})=O(\ex(n,\c{F})^k)=O(\ex(n,\c{F}_{T,k+1}^q)^k).
		\]
		We claim now that
		\begin{equation}
			\ex(n,\c{F}_{T,k+1}^q)^k=O(n^{k+\frac{2k^2}{d}}).\label{eq:diameter and leaves}
		\end{equation}
		Indeed, this is trivial if $k=0$.  If $k\ge 1$, then because $d\ge 2k$ by~\eqref{equation d geq 2k}, we have by \Cref{family bound given diameter} that
		\begin{equation*}
			\ex(n,\c{F}_{T,k+1}^q)^k=O(n^{k+\frac{k(k-1)}{d-k}})=O(n^{k+\frac{2k^2}{d}})
		\end{equation*}
		with this last step using $d\ge 2k$, and $k-1\le k$.  Thus \eqref{eq:diameter and leaves} always holds.
		
		By our bounds on $d$ and $k$ from \eqref{eq:d bound} and \eqref{equation d geq 2k} we have
		\[
		\frac{2k^2}{d}\le \frac{k^2 (\ep v(T))^{1/3}}{v(T)}\le \ep,
		\]
		which together with \eqref{eq:diameter and leaves} gives $\ex(n,T,\c{F})=O(n^{k+\ep})$ as desired.
	\end{proof}
	
	We now prove our final application.
	
	\begin{proof}[Proof of \Cref{non-integral exponents}]
		Given $r\in\mathbb{R}_{\geq 0}\setminus \mathbb{Z}$, we may write $r=\lfloor r\rfloor+2\epsilon$ for $0<\epsilon<1/2$. We claim that for every tree $T$ with $v(T)\ge \ep^{-1} r^3$ there exists no family $\c{F}$ with $\ex(n,T,\c{F})=\Theta(n^r)$. Indeed, assume for contradiction that such a $T,\c{F}$ existed.  Since $r\le (\ep v(T))^{1/3}$, \Cref{diameter corollary} implies that there exists some integer $k$ with $\ex(n,T,\c{F})$ lying between $\Omega(n^{k})$ and $O(n^{k+\ep})$.  But $\Theta(n^r)$ does not lie within this range for any integer $k$, a contradiction.  We conclude that only the finitely many trees $T$ with $v(T)< \ep^{-1} r^3$ can have $\ex(n,T,\c{F})=\Theta(n^r)$.
	\end{proof}

	\section{Generalizaed \texorpdfstring{Tur\'an}{Turan} Preliminaries}\label{section generalized Turan preliminaries}
	Here we establish some preliminary results needed to prove our generalized Tur\'an Theorems~\ref{k-1 edges for trees} and \ref{k-2 edges for paths}, the most important of which is our Edge Helly Theorem, \Cref{edge Helly}.

	\begin{proof}[Proof of \Cref{edge Helly}]
		Recall that we wish to prove that if $T$ is a tree, if $\c{T}$ is a collection of subtrees of $T$, and if $k\ge 0$ is an integer; then if for all subcollections $\c{T}'\sub \c{T}$ of size at most $k+1$ there exists a set of at most $k$ edges $E'\sub E(T)$ which pierces $\c{T}'$ (meaning every $T'\in \c{T}'$ has a vertex used in one of the edges of $E'$), then there exists a set of at most $k$ edges $E\sub E(T)$ which pierces $\c{T}$.
		
		Assume for contradiction that this is false for some $T,\c{T},k$, and among all such counterexamples we choose one with $v(T)$ as small as possible and subject to this $k$ as small as possible. The result is trivial for $k=0$ and if $v(T)=1$ or $v(T)=2$, so we may assume that $k\ge 1$ and $v(T)\ge 3$. Let $x$ be any leaf of $T$ and $y$ the unique neighbor of $x$. Let $T_x$ denote the (single vertex) subtree with $V(T_x)=\{x\}$. We break our argument up into two cases based on if $T_x$ is in $\c{T}$.
		
		\textbf{Case 1:} $T_x\in \c{T}$. Define $\c{T}^-_{x,y}:=\{T'\in \c{T}:x,y\notin V(T')\}$. We claim that every $\c{T}'\sub \c{T}^-_{x,y}$ of size at most $k$ can be pierced by a set of at most $k-1$ edges.  Indeed, for any such collection we have by hypothesis that $\{T_x\}\cup \c{T}'$ can be pierced by a set $E'$ of at most $k$ edges. One of these edges must necessarily be $xy$ since this is the unique edge containing $x$, and since $xy$ is disjoint from every element of $\c{T}'\sub \c{T}^-_{x,y}$, we have that $E'\sm \{xy\}$ is a set of at most $k-1$ edges piercing $\c{T}'$.  This proves the claim, and because of our choice of minimal counterexample, this claim implies that $\c{T}^-_{x,y}$ can be pierced by some set $E$ of at most $k-1$ edges.  It is easy to check then that $E\cup \{xy\}$ is a set of at most $k$ edges which pierces $\c{T}$, proving the result.
		
		\textbf{Case 2:} $T_x \notin \c{T}$. In this case we define a new collection $\c{T}_{-x}=\{T'-x:T'\in \c{T}\}$ which we view as a subtree system of $T-x$.  We claim that every $\c{T}'\sub \c{T}_{-x}$ of size at most $k+1$ can be pierced by a set of at most $k$ edges of $T-x$. Indeed, if say $\c{T}'=\{T_1-x,\ldots,T_\ell-x\}$ for some $\ell\le k+1$, then by hypothesis we know there exists a set $E'$ of at most $k$ edges of $T$ which pierce $\{T_1,\ldots,T_\ell\}$.  If $E'\sub T-x$ then $E'$ will also pierce $\c{T}'$, so the only possible issue is if $xy\in E'$.  In this case if we let $x'$ denote any neighbor of $y$ in $T$ other than $x$ (which exists since $y$ is adjacent to a leaf and $v(T)\ge 2$), then the edge set $(E'\cup \{x'y\})\sm \{xy\}$ will pierce $\c{T}'$ since none of these trees contain $x$ by assumption, proving the claim.
		
		Since we are working with a minimal counterexample, we conclude that there exists some set $E$ of at most $k$ edges of $T-x$ which pierces all of $\c{T}_{-x}$.  Because $T_x=\{x\}\notin \c{T}$ these edges $E$ will pierce $\c{T}$ as well.
	\end{proof}

	Our remaining preliminaries will be written in terms of general graphs $H$, though we will only apply these in the case when $H$ is a tree.  We establish some results around generalized Tur\'an numbers, beginning with an observation which motivates the technical definition of $\c{F}_{H,k}^q$ and which can be viewed as a weak version of our main result \Cref{k-1 edges for trees}.
	
	\begin{proposition}\label{proposition:keyObservation}
		For every graph $H$, integer $k\ge 0$, and family of graphs $\c{F}$, we either have $\ex(n,H,\c{F})=\Om_{H,k}(n^{k+1})$ or there exists an integer $q$ such that every $\c{F}$-free graph is also $\c{F}_{H,k+1}^q$-free.  In particular, either $\ex(n,H,\c{F})=\Om_{H,k}(n^{k+1})$ or $\ex(n,H,\c{F}) \leq \ex(n,H,\c{F}_{H,k+1}^\pow)$.
	\end{proposition}
	\begin{proof}
		Let $H$, $k \geq 0$, and $\c{F}$ be given. For $R \subseteq V(H)$, let $w(R)$ denote the number of connected components of $H - R$. We consider two cases.
		
		\textbf{Case 1:} for every $R \subsetneq V(H)$ with $w(R) \geq k+1$, there exists some integer $\pow(R)$ such that the family $\c{F}$ contains a graph which is a subgraph of $H_R^{\pow(R)}$.  Letting $\pow:=\max_R \pow(R)$, this means that any $\c{F}$-free graph is also $\c{F}_{H,k+1}^\pow$-free as desired.
		
		\textbf{Case 2:} there exists $R \subsetneq V(H)$ with $w(R) \geq k+1$ such that $\c{F}$ does not contain any subgraph of $H_R^\pow$ for any $\pow$. Let $n' = \lfloor{n/|V(H)|}\rfloor$ and consider $G = H_R^{n'}$.  This graph has at most $n$ vertices and is $\c{F}$-free by hypothesis.  Moreover, $H_R^{n'}$ contains at least $(n')^{k+1}=\Om_{H,k}(n^{k+1})$ copies of $H$, namely by taking the copies which use $R$ and for each of the at least $k+1$ components of $H-R$ chooses one of the $n'$ copies making up $H_R^{n'}$ to embed into.
		This implies $\ex(n,H,\c{F}) = \Om(n^{k+1})$, proving the result. 
	\end{proof}
	
	To prove our main results, it will be convenient to shift from counting copies of $H$ in $G$ and instead count \emph{monomorphisms} of $H$ into $G$.
	
	\begin{definition}
		Given graphs $H$ and $G$, we say that a map  $\phi:V(H)\to V(G)$ is a \textbf{monomorphism} of $H$ into $G$ if $\phi$ is both injective and a homomorphism.  We let $\Mon(H,G)$ denote the set of monomorphisms of $H$ into $G$ and we let $\mon(H,G):=|\Mon(H,G)|$. Given a monomorphism $\phi\in \Mon(H,G)$ the \textbf{graph-image} of $\phi$ is the copy $H^*$ of $H$ in $G$ with vertex set $V(H^*)=\phi(V(H))$ and such that $xy\in E(H)$ if and only if $\phi(x)\phi(y)\in E(H^*)$.
	\end{definition}
	
	The following shows that counting monomorphisms is equivalent to counting copies.
	
	\begin{observation}\label{observation monomorphisms same as copies}
		For every pair of graphs $H$ and $G$, the number of copies of $H$ in $G$ equals
		\[
		\frac{\mon(H,G)}{\operatorname{aut}(H)},
		\]
		where $\operatorname{aut}(H)$ denotes the size of the automorphism group of $H$.
	\end{observation}
	This follows from the fact that a copy of $H$ in $G$ is a subgraph of $G$ which is isomorphic to $H$, each such copy giving rise to exactly $\mathrm{aut}(H)$ monomorphisms.  Because of this observation, we will often bound the number of copies of a graph $H$ by instead bounding $\mon(H,G)$. 
	
	The following concept will be useful to work with, where here the set of maps $\Phi$ we use will always be collections of monomorphisms.
	
	\begin{definition}
		We say that a set of maps $\Phi$ with the same domain $S$ is \textbf{distinguishing} if for all distinct $x,y\in S$, we have $\phi(x)\ne \phi'(y)$ for all $\phi,\phi'\in \Phi$.  Equivalently $\{\phi(x):\phi\in \Phi\}\cap \{\phi(y):\phi\in \Phi\}=\emptyset$.
	\end{definition}
	Note that if $\Phi$ is distinguishing then so is any subset $\Phi'\sub \Phi$ as well as its set of restrictions $\{\phi|_{S'}:\phi\in \Phi\}$ for any $S'\sub S$.  The main facts we'll need about these are the following.
	\begin{observation}\label{distinguishing facts}
		Let $H,G$ be graphs and $\Phi\sub \Mon(H,G)$ a distinguishing set of monomorphisms.
		\begin{enumerate}[label=(\alph*)]
			\item For all $X\sub V(H)$, if $\phi,\phi'\in \Phi$ satisfy $\phi(X)=\phi'(X)$, then $\phi(\hat{x})=\phi'(\hat{x})$ for all $\hat{x}\in X$.\label{item distinguishing sets}
			\item There exists a partition $\bigcup_{\hat{x}\in V(T)} V_{\hat{x}}$ of $V(G)$ such that $\{\phi(\hat{x}):\phi\in \Phi\}\sub V_{\hat{x}}$ for all $\hat{x}\in V(H)$.\label{item distinguishing partition}
		\end{enumerate}
	\end{observation}
	
	\begin{proof}
		For \ref{item distinguishing sets}, given $\hat{x}\in X$, since $\phi(X)=\phi'(X)$, we must have $\phi(\hat{x})=\phi'(\hat{y})$ for some $\hat{y}\in X$, but then since $\Phi$ is distingushing, we have $\hat{x}=\hat{y}$, so $\phi(\hat{x})=\phi'(\hat{x})$.
		
		For \ref{item distinguishing partition}, the sets $\{\phi(\hat{x}):\phi\in \Phi\}$ are disjoint for distinct $\hat{x}$ since $\Phi$ is distinguishing and as such one can partition $V(G)$ in any way that preserves $\{\phi(\hat{x}):\phi\in \Phi\}\sub V_{\hat{x}}$ for all $\hat{x}\in V(H)$.
	\end{proof}
	
	We also have the following useful result.
	
	\begin{lemma}\label{distinguishing large set of maps}
		For any graphs $H,G$, there exists a set of distinguishing monomorphisms $\Phi\sub \Mon(H,G)$ with $|\Phi|\ge v(H)^{-v(H)}\mon(H,G)$.
	\end{lemma}
	
	\begin{proof}
		Randomly partition $V(G)$ into sets $\{V_{\hat{x}}:\hat{x}\in V(T)\}$ by including each vertex in $V(G)$ into each of these sets uniformly and independently at random.  Let $\Phi\sub \Mon(H,G)$ be the collection of $\phi\in \Mon(H,G)$ with the property that $\phi(\hat{x})\in V_{\hat{x}}$ for all $\hat{x}\in V(H)$.  Because $\Pr[\phi\in \Phi]=v(H)^{-v(H)}$, we have $\E[|\Phi|]=v(H)^{-v(H)}\mon(H,G)$ by linearity of expectation, and in particular there exists some choice of $V_{\hat{x}}$ sets such that $|\Phi|\ge v(H)^{-v(H)}\mon(H,G)$.
	\end{proof}

	Finally, we need the following.

	\begin{lemma}\label{lemma:sunflower}
		For every graph $H$ and integer $\pow \ge 2$, there exists some integer $\Pow=\Pow(H,q)$ with $\Pow\ge \pow$ such that if $G$ is a graph and if $\Phi\subseteq \mathrm{Mon}(H,G)$ satisfies $|\Phi|\ge \Pow$, then there exists a set $R\subsetneq V(H)$ and distinct maps $\phi_1,\ldots,\phi_\pow\in \Phi$ such that for all $i\ne j$, we have $\phi_i(\hat{x})=\phi_j(\hat{y})$ if and only if $\hat{x}=\hat{y}$ and $\hat{x}\in R$.
	\end{lemma}
	In other words, the maps $\phi_1,\ldots,\phi_q$ are such that the union of their graph-images is a copy of $H_R^\pow$ in $G$.
	
	\begin{proof}
		This result is a consequence of the Erd\H{o}s-Rado sunflower lemma~\cite{ER1960}, which states\footnote{We note in passing that it is a major open problem to improve upon the bounds of the Erd\H{o}s-Rado sunflower lemma, though we will only need that some finite bound exists.  For a more thorough treatment of this result and best known bounds we refer the reader to~\cite{ALWZ2021}.} that if $\c{H}$ is an $r$-uniform hypergraph with $|E(\c{H})|\ge k!(r-1)^k$, then there exist edges $e_1,\ldots,e_k\in E(\c{H})$ and a set $K\sub V(\c{H})$ such that $e_i\cap e_j=K$ for every $i\ne j$. We will ultimately apply this result with $r:=v(H)$ and $k:=q\cdot v(H)!$, with us in the end proving the lemma for the value $\Pow(H,G):=v(H)!\cdot k!(r-1)^k$.
		
		Given a collection of at least $\Pow(H,G)$ monomorphisms $\Phi\subseteq \mathrm{Mon}(H,G)$, we first take a subcollection $\Phi'\subseteq \Phi$ such that $\phi(V(H))\neq \phi'(V(H))$ for all $\phi,\phi'\in \Phi'$. Since at most $v(H)!$ monomorphisms can have the same image, we can find such a collection with $|\Phi'|\geq |\Phi|/v(H)!\geq k!(r-1)^k$.
		
		Consider the $r$-uniform hypergraph $\mathcal{H}$ with $V(\mathcal{H})=V(G)$ and
		\[
		E(\mathcal{H})=\{\phi(V(H))\mid \phi\in \Phi'\}.
		\]
		Note that $|E(\mathcal{H})|=|\Phi'|$ by the defining property of $\Phi'$. By the sunflower lemma, there are $k$ edges of $\mathcal{H}$, say $\phi_1(V(H)),\dots,\phi_k(V(H))$, and a set $K\subseteq V(G)$ such that $\phi_i(V(H))\cap \phi_j(V(H))=K$ for all $i\neq j$. Let $K=\{v_1,v_2,\dots,v_t\}$, and associate to each $\phi_i$ the tuple $(\phi_i^{-1}(v_1),\phi_i^{-1}(v_2),\dots,\phi_i^{-1}(v_t))$. Since the $\phi_i$'s are injective and $t\leq v(H)$, there are at most $v(H)!$ possible tuples, so by pigeonhole there must exist a collection of at least $k/v(H)!=q$ of the $\phi_i$'s which correspond to the same tuple.  Without loss of generality we can assume that $\phi_1,\dots,\phi_q$ have this property. Let $R:=\phi_1^{-1}(K)$ . This gives for $i,j\in [q]$ with $i\neq j$ that $\phi_i(\hat{x})=\phi_j(\hat{y})$ if and only if $\hat{x}=\hat{y}$ and $\hat{x}\in R$. Furthermore, $R\neq V(H)$ since $\phi_1\neq \phi_2$, but $\phi_1$ and $\phi_2$ agree on $R$. Thus, $R$ along with $\phi_1,\ldots,\phi_q$ satisfy the conditions of the lemma.
	\end{proof}
	
	\section{Proof of \texorpdfstring{Theorem~\ref{k-1 edges for trees}}{Theorem 1.3}}\label{section main theorem proof}
	
	Our main goal for this section will be to prove the following technical result.
	
	\begin{theorem}\label{theorem:treeBoundEG}
		If $T\ne K_1$ is a tree and if $k\geq 0$ and $\pow\ge 2$ are integers, then any $n$-vertex graph $G$ which is $\c{F}_{T,k+1}^\pow$-free contains at most $O(e(G)^{k})$ copies of $T$. 
	\end{theorem}
	
	Our motivation for this is that it quickly implies \Cref{k-1 edges for trees}.
	
	\begin{proof}[Proof of \Cref{k-1 edges for trees}]
		If there does not exist a value of $q$ such that every $\c{F}$-free graph is $\c{F}_{T,k+1}^q$-free, then \Cref{proposition:keyObservation} implies $\ex(n,T,\c{F})=\Om_{T,k}(n^{k+1})$ as desired.
		
		Now let us assume there is such a $q$, and assume for the moment that $q\geq 2$. By taking $G$ to be an $\c{F}$-free graph with $\ex(n,T,\c{F})$ copies of $T$, we note that $G$ is $\mathcal{F}_{T,k+1}^q$-free as well by the defining property of $q$, so \Cref{theorem:treeBoundEG} gives
		\[
		\ex(n,T,\c{F})=O(e(G)^{k})=O(\ex(n,\c{F})^{k}),
		\]
		The only remaining case is $\pow=1$, which follows from the case $\pow=2$ and monotonicity.
	\end{proof}

	\textbf{Proof Intuition for \Cref{theorem:treeBoundEG}}.  The statement of \Cref{theorem:treeBoundEG} suggests how it might be proved, namely by showing that in any $\c{F}_{T,k+1}^\pow$-free graph $G$, every copy of $T$ contains a set of $k$ edges $E$ with the property that $E$ is contained in at most $O(1)$ copies of $T$ in $G$. 
	\begin{figure}
		\centering
		\begin{tikzpicture}[scale=1.2]
			
			\node at (1,0) { $\bullet$};
			\node at (2,0) { $\bullet$};
			\node at (3,0) { $\bullet$};
			\node at (4,0) { $\bullet$};
			\node at (5,0) { $\bullet$};
			\node (u5) at (5,.5) { $\bullet$};
			\node (w5) at (5,-.5) { $\bullet$};
			\node (w1) at (1,-.8) { $\bullet$};
			\node (u1) at (1,.8) { $\bullet$};
			\node (w2) at (2,-.4) { $\bullet$};
			\node (u2) at (2,.4) { $\bullet$};
			\draw (3,0)--(2,0);
			\draw (4,0)--(5,0);
			\draw (1,0)--(2,0);
			\draw (3,0)--(4,0);
			\draw (4,0)--(5,-.5);
			\draw (4,0)--(5,.5);
			\draw (3,0)--(2,.4);
			\draw (3,0)--(2,-.4);
			\draw (2,.4)--(1,.8);
			\draw (2,-.4)--(1,-.8);
		\end{tikzpicture}
		\caption{A depiction of $(P_5)_{\{\hat{x}_3,\hat{x}_4\}}^3$.  Observe that each copy of $P_5$ in this graph can be identified by specifying which edges plays the role of $\hat{x}_1\hat{x}_2$ and $\hat{x}_4\hat{x}_5$.}
		\label{figure P5}
	\end{figure}
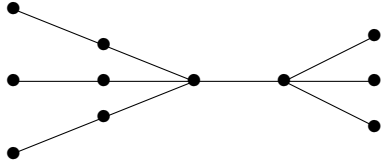
	For example, consider $P_5$ the path graph on $\hat{x}_1\hat{x}_2\hat{x}_3\hat{x}_4\hat{x}_5$ and $G:=(P_5)_{\{\hat{x}_3,\hat{x}_4\}}^n$ (i.e.\ $G$ is obtained by taking $n$ copies of $P_5$ which all agree on the edge $\hat{x}_3\hat{x}_4$ and are otherwise disjoint; see \Cref{figure P5}).   Here each copy $P$ of $P_5$ in $G$ can be uniquely identified by specifying the edges of $G$ which play the role of $\hat{x}_1\hat{x}_2$ and $\hat{x}_4\hat{x}_5$ in $P$, implying that $G$ contains at most $e(G)^2$ copies of $P_5$.  The same argument works\footnote{Technically specifying $\hat{x}_2\hat{x}_3$ and $\hat{x}_4\hat{x}_5$ only specifies all of the ``canonical'' $P_5$'s, i.e.\ the ones that go from left to right.  However, by considering random 5-partitions of the host graph $G$ one can reduce the problem of upper bounding copies of $P_5$ in $G$ with upper bounding ``canonicall'' copies of $P_5$.} if we alternatively specified which edges play the role of $\hat{x}_2\hat{x}_3$ and $\hat{x}_4\hat{x}_5$ in $P$.
	
	In the proof above, it is  critical that we specify the edge of $G$ playing the role of $\hat{x}_4\hat{x}_5$, as even if we specify every other edge of a copy of $P_5$, there would still exist $n$ ways to extend these specified edges into distinct copies of $P_5$ in $G$.  More generally, the edge set $E$ we wish to choose for each copy must ``intersect'' with every subtree which has ``many options'' for being extended.
	
	More precisely, for each monomorphism $\phi:V(T)\to V(G)$, we wish to define a subtree system $\c{T}_\phi$ of $T$ whose elements have ``many extensions'' with respect to $\phi$, in the sense that for each $T'\in \c{T}_\phi$, there exist many monomorphisms $\phi'$ which agree with $\phi$ outside of $T'$ (i.e.\ there exist many ways of changing how $\phi$ maps $T'$ while maintaining that $\phi$ is a monomorphism).  For example, if $G=(P_5)^n_{\{\hat{x}_3,\hat{x}_4\}}$ then for every $\phi$ we will define $\c{T}_\phi$ to consist of the subtrees which either contain $\hat{x}_1\hat{x}_2$ or $\hat{x}_5$, i.e.
	\[
	\c{T}_{\phi}=\{\hat{x}_1\hat{x}_2,\hat{x}_1\hat{x}_2\hat{x}_3,\hat{x}_1\hat{x}_2\hat{x}_3\hat{x}_4,\hat{x}_1\hat{x}_2\hat{x}_3\hat{x}_4\hat{x}_5,\hat{x}_2\hat{x}_3\hat{x}_4\hat{x}_5,\hat{x}_3\hat{x}_4\hat{x}_5,\hat{x}_4\hat{x}_5,\hat{x}_5\}.
	\]
	More generally, if $G=T_R^n$, then $\c{T}_\phi$ will consist of the subtrees that contain a connected component of $T-R$.
	
	The key insight about working with these sets $\c{T}_\phi$ is that whatever edge set $E$ we ultimately choose to identify each copy $\phi$ with, this set $E$ must have the property that every $T'\in \c{T}_\phi$ has $\phi(V(T'))$ intersecting with an edge in $E$.  Indeed, if this were not the case then $E$ would necessarily be contained in ``many'' copies of $T$, namely those which agree with $\phi$ outside of the subtree $T'\in \c{T}_\phi$ which does not intersect with $E$.  Ultimately then, we require a sufficient condition to guarantee that a subtree system $\c{T}$ can be pierced by a small set of edges, which motivates the need for our Edge Helly Theorem, \Cref{edge Helly}.
	
	\subsection{Formal Details}
	
	We begin by collecting some technical definitions that will be used throughout this section and the next. Our first definitions precisely characterize which subtrees $T'\subseteq T$ have ``many extensions'', as mentioned in the proof sketch above.
	
	\begin{definition}\label{definition pseudo piercing}
		Let $T$ be a tree, $G$ a graph, $\Phi\subseteq \Mon(T,G)$ a collection of monomorphisms and $C>0$ a constant. Given a subtree $T'\sub T$, define the  \textbf{neighborhood of $T'$}
		\[
		N(T'):=\{\hat{x}\in V(T)\sm V(T'):N(\hat{x})\cap V(T')\ne \emptyset\}
		. \]
		For $S\sub V(G)$ we define the \textbf{extendable set}
		\[
		\Psi(T',S;\Phi):=\{\phi|_{V(T')}:\phi\in \Phi,\ \phi(N(T'))=S\}.
		\]
		That is, $\Psi$ is the restriction to $T'$ of all maps $\phi\in \Phi$ which map $N(T')$ to $S$, i.e.\ these are maps of $T'$ which can be extended to a map of $T$ while acting on $N(T')$ in a specified way. Finally, If $\phi\in \Phi$, then we define the \textbf{set of highly extendable subtrees}, $\c{T}_{\phi,\Phi,C}$ to be the set of subtrees $T'\sub T$ such that $|\Psi(T',\phi(N(T'));\Phi)|\ge C$.
	\end{definition}
	
	Our next set of definitions gives a generalization of the notion of piercing which may be of independent interest.
	
	\begin{definition}
		Given a tree $T$, a set of distinguishing monomorphisms $\Phi\sub \Mon(T,G)$, and some $\phi\in \Phi$, we say a set of vertices $\widehat{W}\sub V(T)$ \textbf{$(\phi,\Phi)$-pseudo-pierces} a subtree $T'\sub T$ if there exists a vertex $\hat{x}\in V(T')$ such that 
		\[
		|\{\phi'(\hat{x}):\phi'\in \Phi,\ \phi'(\widehat{W})=\phi(\widehat{W})\}|\le v(T).
		\]
		Similarly, we say $\widehat{W}$ $(\phi,\Phi)$-pseudo-pierces a subtree system $\c{T}$ if it $(\phi,\Phi)$-pseudo-pierces every $T'\in \c{T}$.

		Given a set of edges $\widehat{E}\subseteq E(T)$, we let $V(\widehat{E})$ denote the set of vertices of $T$ incident to an edge of $\widehat{E}$.  If $\widehat{E}$ is a set of at most $a$ edges and if $\widehat{U}\subseteq V(T)$ a set of at most $b$ vertices, we say that $(\widehat{E},\widehat{U})$ is an \textbf{$(a,b,\phi,\Phi)$-pseudo-piercing set} for $\mathcal{T}$ if $V(\widehat{E})\cup \widehat{U}$ $(\phi,\Phi)$-pseudo-pierces $\c{T}$.  Similarly we say that $(\widehat{E},\widehat{U})$ is an \textbf{$(a,b)$-piercing set} for $\c{T}$ if $V(\widehat{E})\cup \widehat{U}$ pierces $\c{T}$.
	\end{definition}
	It will be useful to record the easy fact that piercing implies pseudo-piercing which we will implicitly use throughout.
	\begin{lemma}
		If $T$ is a tree and if $\widehat{W}\sub V(T)$ pierces some subtree $T'\sub T$, then $\widehat{W}$ $(\phi,\Phi)$-pseudo-pierces $T'$ for any set of distinguishing monomorphisms $\Phi\sub \Mon(T,G)$ and $\phi\in \Phi$.  In particular, $(a,b)$-piercing sets are $(a,b,\phi,\Phi)$-pseudo-piercing sets.
	\end{lemma}
	Indeed, this follows because if $\Phi$ is distinguishing, we have by \Cref{distinguishing facts}(a) that any $\hat{x}\in \widehat{W}\cap V(T')$ (which exists by assumption of piercing) has
	\[|\{\phi'(\hat{x}):\phi'\in \Phi,\ \phi'(\widehat{W})=\phi(\widehat{W})\}|\le |\{\phi'(\hat{x}):\phi'\in \Phi,\ \phi'(\hat{x})=\phi(\hat{x})\}|=1,\]
	proving this $(\phi,\Phi)$-pseudo-pierces $T'$.
	
	The notion of psuedo-piercing is not needed to prove our main result of this section, \Cref{theorem:treeBoundEG}.  However, it will be necessary in \Cref{section refinement for paths}, and so we present our results here in terms of the general notion of pseudo-piercing so that we may use them later on and in potential future works.
	
	As hinted at in the proof intuition above, our goal will be to show that the families $\c{T}_{\phi,\Phi,C}$ can be pierced by a small number of edges using the Edge Helly Theorem.  We will then use this in conjunction with the following to show that this implies $G$ contains few copies of $T$.
	
	\begin{proposition}\label{piercing equals few copies}
		Let $T\ne K_1$ be a tree, $G$ a graph, $\Phi\sub \Mon(T,G)$ a set of distinguishing monormphisms, and $C,a,b$ non-negative integers.  If  $\Phi'\sub \Phi$ is such that for every $\phi\in \Phi'$, the collection $\c{T}_{\phi,\Phi,C}$ has an $(a,b,\phi,\Phi)$-pseudo-piercing set, then
		\[
		|\Phi'|=O_C(e(G)^a v(G)^b).
		\]
	\end{proposition}
	
	\begin{proof}
		The result is trivial if $e(G)=0$ since $e(T)\ge 1$, so we may assume $e(G)>0$ from now on.  For each $\phi \in \Phi$, let $\widehat{E}_\phi,\widehat{U}_\phi$ be the set of at most $a$ edges and at most $b$ vertices of $T$ such that $V(\widehat{E}_\phi)\cup \widehat{U}_\phi$ is a $(\phi,\Phi)$-pseudo-piercing set for $\mathcal{T}_{\phi,\Phi,C}$, and define
		\[
		E_\phi:=\{\phi(\hat{x})\phi(\hat{y}):\hat{x}\hat{y}\in \widehat{E}_\phi\},
		\] 
		and 
		\[U_\phi:=\phi(\widehat{U}_\phi)=\{\phi(\hat{x}):\hat{x}\in \widehat{U}_\phi\}.\]
		Observe that $E_\phi$ is a set of at most $a$ edges of $G$ since $\phi$ is a homomorphism and similarly that $U_\phi$ is a set of at most $b$ vertices of $G$. 
		
		We aim to show that for each set of at most $a$ edges $E\sub E(G)$ and each set of at most $b$ vertices $U\sub V(G)$, the number of $\phi\in \Phi'$ with $E_{\phi}=E$ and $U_\phi=U$ is $O(1)$, from which it will follow that $|\Phi'|=O(e(G)^{a}v(G)^b)$ since there are at most $O(e(G)^av(G)^b)$ choices for $E,U$ (here we implicitly use the assumption $e(G),v(G)>0$).
		
		From now on we fix some $E,U$ as above.	  Define \[\Phi'_{E,U}:=\{\phi\in \Phi': V(E)\cup U\sub \phi(V(T))\},\] and 
		\[
		\widehat{V}_{E,U}:=\{\hat{x}\in  V(T): |\{\phi(\hat{x}):\phi\in \Phi_{E,U}'\}|\le v(T)\}.
		\]    
		The motivation for this is the following.
		
		\begin{claim}
			If $\phi\in \Phi'$ is such that $E_\phi=E$ and $U_\phi=U$, then for every $T'\in \c{T}_{\phi,\Phi,C}$, we have $V(T')\cap \widehat{V}_{E,U}\ne \emptyset$.
		\end{claim}
		
		\begin{proof}
			By hypothesis of the proposition, for every $T'\in \c{T}_{\phi,\Phi,C}$ there exists $\hat{x}\in V(T')$ such that 
			\begin{equation}\label{equation claim 4.4 interpreting piercing}
				\left|\left\{\psi(\hat{x}):\psi\in \Phi,\ \psi\left(V(\widehat{E}_\phi)\cup \widehat{U}_\phi\right)=\phi\left(V(\widehat{E}_\phi)\cup \widehat{U}_\phi\right)\right\}\right|\le v(T).
			\end{equation}
			Note that $\phi\left(V(\widehat{E}_\phi)\cup \widehat{U}_\phi\right)=V(E_\phi)\cup U_\phi=V(E)\cup U$. This, \eqref{equation claim 4.4 interpreting piercing}, and the fact that $\Phi'\subseteq \Phi$ gives
			\begin{equation}\label{equation claim 4.4 interpreting piercing 2}
				\left|\left\{\psi(\hat{x}):\psi\in \Phi',\ \psi\left(V(\widehat{E}_\phi)\cup \widehat{U}_\phi\right)=V(E)\cup U\right\}\right|\le v(T).
			\end{equation}
			Since $\Phi'$ is distinguishing,  for any $\psi\in \Phi'$, having $\psi(V(\widehat{E}_\phi)\cup \widehat{U}_\phi)=V(E)\cup U$ is equivalent to having $V(E)\cup U\sub \psi(V(T))$.  Thus~\eqref{equation claim 4.4 interpreting piercing 2} is equivalent to $|\{\psi(\hat{x}):\psi\in \Phi'_{E,U}\}|\le v(T)$, which means that $\hat{x}\in \widehat{V}_{E,U}$ by definition.
		\end{proof}
		
		Let $T'_1,\ldots,T'_m$ denote the connected components of $T-\widehat{V}_{E,U}$.  The final fact we need is the following.
		
		\begin{claim}\label{cl:ETSmall}
			If there exists some $\phi\in \Phi'$ such that $E_\phi=E$ and $U_\phi=U$, then for all $j$,
			\[
			|\Psi(T'_j,\phi(N(T'_j));\Phi)|<C.
			\]
		\end{claim}
		
		\begin{proof}
			Assume for contradiction that this did not hold for some $\phi,j$.  By definition this implies $T'_j\in \c{T}_{\phi,\Phi,C}$, but by the previous claim this implies $V(T'_j)\cap \widehat{V}_{E,U}\ne \emptyset$.  This is impossible since $V(T'_j)\sub V(T)\sm \widehat{V}_{E,U}$, giving the desired contradiction.
		\end{proof}
		
		We ultimately aim to show that the number of $\phi\in \Phi'$ with $E_\phi=E$ and $U_\phi=U$ is at most $(C\cdot v(T))^{v(T)}=O(1)$, which will give the desired result. This is certainly true if no such $\phi$ exists, so we may assume that at least one such $\phi$ exists.  Moreover, we observe that every such $\phi$ must have $\phi\in \Phi'_{E,U}$ since this is the only possible way one can have $E_\phi=E$ and $U_\phi=U$ by definition of these sets.
		
		For each such $\phi$, let $\phi_j:=\phi|_{V(T'_j)}$ for $j\ge 1$ and let $\phi_0:=\phi|_{\widehat{V}_{E,U}}$.  Observe that $\phi$ is uniquely determined by the tuple $(\phi_0,\phi_1,\ldots,\phi_m)$ since every $\hat{x}\in V(T)$ is either in $\widehat{V}_{E,U}$ or $V(T'_j)$ for some $j\geq 1$. Note that the number of choices for $\phi_0$ is at most $v(T)^{|\widehat{V}_{E,U}|}\le v(T)^{v(T)}$ by the definition of $\widehat{V}_{E,U}$ and the fact that $\phi\in \Phi'_{E,U}$.  To count the choices for $\phi_j$ with $j\ge 1$ given $\phi_0$, we first observe that $N(T'_j)\sub \widehat{V}_{E,U}$ by definition of $T'_j$ being a connected component of $T-\widehat{V}_{E,U}$.  It follows then that $\phi(N(T'_j))=\phi_0(N(T'_j))$, so given $\phi_0$ the set $\Psi(T'_j,\phi(N(T'_j));\Phi)$ is determined.  Moreover, $\phi_j$ is an element of $\Psi(T'_j,\phi(N(T'_j));\Phi)$  since $\phi_j$ equals $\phi|_{V(T'_j)}$ with $\phi\in \Phi$.  By \Cref{cl:ETSmall}, there are at most $C$ choices for $\phi_j$ in this set, and hence at most $C^{m}\leq  C^{v(T)}$ choices of the vectors $\phi_1,\ldots,\phi_m$ given $\phi_0$.  In total then the number of choices for the vector $(\phi_0,\phi_1,\ldots,\phi_m)$ encoding $\phi$ is at most $(C\cdot v(T))^{v(T)}=O(1)$.  Putting all of this together gives the desired bound of $|\Phi'|=O(e(G)^{a}v(G)^b)$.
	\end{proof}
	
	It remains to find conditions for which $\c{T}_{\phi,\Phi,C}$ can be pierced by a small number of edges.  For this we will utilize the following essentially equivalent version of \Cref{edge Helly}.
	
	\begin{lemma}\label{lemma edge helly lemma}
		Let $T\ne K_1$ be a tree, $\c{T}$ a subtree system of $T$ and $k\geq 0$ an integer.  If for every collection $T'_1,\ldots,T'_{k+1}\in \c{T}$, there exists some $i\neq j$ such that $(V(T'_i)\cup N(T'_i))\cap V(T'_j)\neq \emptyset$, then there exists a set of at most $k$ edges that pierces $\mathcal{T}$.
	\end{lemma}
	
	\begin{proof}
		If $k=0$, the statement is vacuously true, so assume $k\geq 1$. Let $\mathcal{T}':=\{T'_1,\ldots,T'_{k+1}\}$ be a subcollection of $\mathcal{T}$. Our goal is to find a set of $k$ edges that pierces $\mathcal{T}'$. Let $i,j$ be such that $(V(T'_i)\cup N(T'_i))\cap V(T'_j)\neq \emptyset$ and assume without loss of generality that $i=k$ and $j=k+1$. Let $\hat{u}\in (V(T'_{k})\cup N(T'_{k}))\cap V(T'_{k+1})$, and let $\hat{e}_{k}$ be any edge containing $\hat{u}$ that intersects $V(T_{k}')$, which exists since $\hat{u}\in V(T'_{k})\cup N(T'_{k})$. Note that $\hat{e}_{k}$ intersects both $V(T_{k})$ and $V(T_{k+1})$.
		
		Now for each $i\in [k-1]$, let $\hat{e}_i$ be an edge intersecting $V(T_i)$. Then $\{\hat{e}_1,\hat{e}_2,\dots,\hat{e}_{k}\}$ pierces $\mathcal{T}'$. Thus by \Cref{edge Helly}, there is a set of at most $k$ edges which pierces $\mathcal{T}$.
	\end{proof}
	
	With this we can prove the following.
	
	\begin{proposition}\label{GT Pierces by Edges}
		Let $T\ne K_1$ be a tree and $q \ge 2$ an integer. Then there exists a contant $C_0=C_0(T,q)$ depending only on $T$ and $q$ such that the following holds:
		
		If $k\geq 0$ and $C\geq C_0$ are integers, $G$ is a $\mathcal{F}_{T,k+1}^q$-free graph, and $\Phi\sub \Mon(T,G)$ a set of distinguishing monomorphisms, then for each $\phi\in \Phi$ there exists a set of at most $k$ edges $\widehat{E}_\phi\sub E(T)$ such that $\widehat{E}_\phi$ pierces $\mathcal{T}_{\phi,\Phi,C}$.
	\end{proposition}
	
	\begin{proof}
		We will prove this result with $C_0:=\max_{T'} \Pow(T',q)$ where $\Pow(T',q)$ is the constant from \Cref{lemma:sunflower} and the maximum ranges over all subtrees $T'\sub T$. Fix some $\phi\in\Phi$. By \Cref{lemma edge helly lemma} we may assume for contradiction that there exist $T'_1,\ldots,T'_{k+1}\in \c{T}_{\phi,\Phi,C}$ with $(V(T'_i)\cup N(T'_i))\cap V(T'_j)=\emptyset$ for all $i\ne j$.  
		
		Since $T'_j\in \c{T}_{\phi,\Phi,C}$, we have $|\Psi(T'_j,\phi(N(T'_j)),\Phi)|\geq C\geq C_0\geq\Pow(T'_j,q)$.  By \Cref{lemma:sunflower} and the fact that each element of $\Psi(T'_j,\phi(N(T'_j)),\Phi)$ is a monomorphism from $T'_j$ to $G$, there exist distinct monomorphisms $\psi_j^{(1)},\ldots,\psi_j^{(q)}\in \Psi(T'_j,\phi(N(T'_j)),\Phi)$ as well as a set of vertices $R_j\sub V(T'_j)$ such that $\psi_j^{(p)}(\hat{u})=\psi_{j}^{(p')}(\hat{u})$ for $p\ne p'$ if and only if $\hat{u}\in R_j$, and the images of $\psi_j^{(p)}$ and $\psi_j^{(p')}$ are disjoint outside of $R_j$.
		
		For convenience let $V_j:=V(T'_j)$ and define $V_0:=V(T)\setminus \bigcup_{j\ge 1} V_j$, noting that these sets partition $V(T)$ since we in particular assumed $V(T'_i)\cap V(T'_j)=\emptyset$ at the start of the proof.  We will make use of the following observation.
		\begin{claim}\label{claim each N(T) is in V_0}
			We have $\bigcup_{j\ge 1} N(T'_j)\sub V_0$.
		\end{claim}
		
		\begin{proof}
			We have that $N(T'_j)$ is disjoint from $V_j$ by definition, and it is also disjoint from $V_{j'}$ for all other $j'\ge 1$ by our assumption that $(V(T'_j)\cup N(T'_j))\cap V(T'_{j'})=\emptyset$ for all $j\ne j'$, so $N(T'_j)$ must lie entirely in $V_0=V(T)\setminus \bigcup_{j'\ge 1} V_{j'}$.
		\end{proof}
		
		For convenience define $\psi_0^{(p)}:=\phi|_{V_0}$ for all $1\le p\le q$, and note that each $\psi_j^{(p)}$ is the restriction of some monomorphism in $\Phi$ to $V_j$. For $1\le p\le q$ let $\phi^{(p)}:V(T)\to V(G)$ be the map such $\phi^{(p)}|_{V_j}=\psi_j^{(p)}$ for all $j\ge 0$.  We ultimately aim to show that the maps $\phi^{(1)},\ldots,\phi^{(q)}$ define an element of $\c{F}_{T,k}^q$ in $G$.
		
		\begin{claim}\label{claim each map is a monomorphism}
			Each map $\phi^{(p)}$ is a monomorphism.
		\end{claim}
		\begin{proof}
			Fix $p$. Note that for each $\hat{x}\in V(T)$ we have $\phi^{(p)}(\hat{x})=\phi'(\hat{x})$ for some $\phi'\in \Phi$; indeed if $\hat{x}\in V_j$ then $\phi^{(p)}(\hat{x})=\psi_j^{(p)}(\hat{x})$, and $\psi_j^{(p)}$ is by definition the restriction of some map in $\Phi$.  Because $\Phi$ is distinguishing, this implies that $\phi^{(p)}$ is injective, so it remains to prove that $\phi^{(p)}$ is a homomorphism.
			
			Let $\hat{x}\hat{y}$ be an arbitrary edge of $T$. If $\hat{x},\hat{y}\in V_j$ for some $j$, then $\phi^{(p)}(\hat{x})\phi^{(p)}(\hat{y})=\psi_j^{(p)}(\hat{x})\psi_j^{(p)}(\hat{x})$ is an edge of $G$ since $\psi_j^{(p)}$ is a monomorpism. Thus, we may assume $\hat{x}\in V_j$ and $\hat{y}\in V_{j'}$ for $j\neq j'$, and without loss of generality we may assume $j\geq 1$.
			
			Since $\hat{x}\in V_j$ and $\hat{y}\not\in V_j$, we have that $\hat{y}\in N(T_j')$. Furthermore, \Cref{claim each N(T) is in V_0} gives us that $N(T'_j)\sub V_0$, so $\phi^{(p)}(\hat{y})=\psi_0^{(p)}(\hat{y})=\phi(\hat{y})$. Because $j\ge 1$, we have $\psi_j^{(p)}\in \Psi(T'_j,\phi(N(T'_j)),\Phi)$, which by definition means there exists some $\phi'\in \Phi$ with $\phi'|_{V_j}=\psi_j^{(p)}$ and with $\phi'(N(T'_j))=\phi(N(T'_j))$.  In particular, we must have $\phi'(\hat{y})=\phi(\hat{y})$ because $\hat{y}\in N(T'_j)$ and $\Phi$ is distinguishing.  In total, we have
			\[
			\phi^{(p)}(\hat{x})\phi^{(p)}(\hat{y})=\psi_j^{(p)}(\hat{x})\phi(\hat{y})=\phi'(\hat{x})\phi'(\hat{y})\in E(G).
			\]
		\end{proof}
		
		Define \[R:=V_0\cup \bigcup_{j\ge 1} R_j.\]
		\begin{claim}\label{claim amalgamation}
			The following holds:
			\begin{enumerate}[label=(\roman*)]
				\item For all $p\ne p'$, we have $\phi^{(p)}(\hat{x})=\phi^{(p')}(\hat{y})$ if and only if  $\hat{x}=\hat{y}$ and $\hat{x}\in R$.\label{claim amalgamation (ii'') well-defined}
				\item $T-R$ has at least $k+1$ connected components.\label{claim amalgamation (iii'') T-R has k components}
			\end{enumerate}
		\end{claim}
		\begin{proof}
			For the reverse direction of~\ref{claim amalgamation (ii'') well-defined}, if $\hat{x}=\hat{y}$ lies in some $R_j\sub V_j$, then by definition of $R_j$,  
			\[
			\phi^{(p)}(\hat{x})=\psi_j^{(p)}(\hat{x})=\psi_{j}^{(p')}(\hat{x})=\phi^{(p')}(\hat{x}).
			\] 
			If instead $\hat{x}=\hat{y}\in V_0$, then
			\[
			\phi^{(p)}(\hat{x})=\phi(\hat{x})=\phi^{(p')}(\hat{x}).
			\] 
			
			Now, for the forward direction of~\ref{claim amalgamation (ii'') well-defined}, assume that $\phi^{(p)}(\hat{x})=\phi^{(p')}(\hat{y})$. Note that by the definition of $\phi^{(p)}$ and $\phi^{(p')}$, there exists $\varphi,\varphi'\in\Phi$ such that $\phi^{(p)}(\hat{x})=\varphi(\hat{x})$ and $\phi^{(p')}(\hat{y})=\varphi'(\hat{y})$, and thus since $\Phi$ is distinguishing, we must have that $\hat{x}=\hat{y}$. 
			
			If $\hat{x}\in V_0\subseteq R$, we are done, so assume $\hat{x}\in V_j$ for some $j\geq 1$. In this case, we have that 
			\[
			\psi_j^{(p)}(\hat{x})=\phi^{(p)}(\hat{x})=\phi^{(p')}(\hat{x})=\psi_{j}^{(p')}(\hat{x}),
			\]
			and thus $\hat{x}\in R_j\subseteq R$ by the definition of $R_j$. 
			
			For~\ref{claim amalgamation (iii'') T-R has k components}, let $\hat{x}_j$ be an arbitrary vertex in $V(T'_j)\sm R_j$ for each $j\ge 1$, which exists since $R_j$ is a proper subset of $V(T'_j)$. Now, given $j\neq j'$, if $P$ is the path in $T$ from $\hat{x}_j$ to $\hat{x}_{j'}$, then $P$ starts in $V(T_j)$ and ends outside $V(T_j)$, so in particular $P$ must contain a vertex in $N(T_j)$, and so by \Cref{claim each N(T) is in V_0}, $P$ contains a vertex in $V_0\subseteq R$. Thus, $\hat{x}_j$ and $\hat{x}_{j'}$ are in different components of $T-R$, giving at least $k+1$ components of $T-R$.
		\end{proof}
		In total, \Cref{claim each map is a monomorphism} gives us that each $\phi^{(p)}$ is a monomorphism, and Claim~\ref{claim amalgamation} gives us that the graph-images of the $\phi^{(p)}$'s constitute a copy of $T_R^\pow\in \c{F}_{T,k+1}^\pow$. Thus contradicts the fact that $G$ is $\c{F}_{T,k+1}^\pow$-free, proving the result.
	\end{proof}
	
	We now put these pieces together to prove the main result of this section.
	\begin{proof}[Proof of \Cref{theorem:treeBoundEG}]
		Let $G$ be an $n$-vertex $\c{F}_{T,k+1}^\pow$-free graph, and let $C:=C_0(T,q)$ be the constant guaranteed by \Cref{GT Pierces by Edges}. By \Cref{distinguishing large set of maps} there exists a set of distinguishing monomorphisms $\Phi\sub \Mon(T,G)$ with $|\Phi|=\Theta(\mon(T,G))$. Then by \Cref{GT Pierces by Edges}, we have that for every $\phi\in \Phi$, there exists a set of at most $k$ edges $\widehat{E}_\phi\sub E(T)$ such that every $T'\in \c{T}_{\phi,\Phi,C}$ contains a vertex $\hat{x}\in V(\widehat{E}_\phi)$.  In particular, $\c{T}_{\phi,\Phi,C}$ has a $(k,0,\phi,\Phi)$-pseudo-piercing set. Then, by \Cref{piercing equals few copies} (applied with $\Phi'=\Phi$), we have that $|\Phi|=O(e(G)^{k})$, and thus $\mon(T,G)=O(e(G)^{k})$. The result then follows from \Cref{observation monomorphisms same as copies}.
	\end{proof}

	\section{Improvement for Paths - Proof of \texorpdfstring{Theorem~\ref{k-2 edges for paths}}{Theorem 1.10}}\label{section refinement for paths}
	
	In this section we prove the following which, analogous to how \Cref{theorem:treeBoundEG} implied \Cref{k-1 edges for trees} in the previous section, will imply \Cref{k-2 edges for paths}.
	
	\begin{theorem}\label{path reduce by 1}
		If $k,t\ge 3$ are integers with  $t\not \equiv 2 \mod k-1$, then any $\c{F}_{P_t,k+1}^q$-free graph $G$ has at most $O_q(e(G)^{k-1} v(G))$ copies of $P_t$.
	\end{theorem}
	
	The main difficulty in proving \Cref{path reduce by 1} is that the natural analog of \Cref{GT Pierces by Edges} (i.e.\ that if some set $\c{T}_{\phi,\Phi,C}$ can not be pierced by $k-1$ edges and 1 vertex, then $G$ contains an element of $\c{F}_{P_t,k+1}^q$) is not true.  To get around this, we will instead prove that if $G$ contains many copies of $T$ and if some set $\c{T}_{\phi,\Phi,C}$ can not be \textit{pseudo-pierced} by $k-1$ edges and 1 vertex, then $G$ contains an element of $\c{F}_{P_t,k+1}^q$.
	
	\subsection{Refinement sequences and \texorpdfstring{$k$}{k}-nice tuples}
	
	We begin with a technical refinement of \Cref{distinguishing large set of maps}.  Given a subtree system $\mathcal{T}$, we say that a subtree $T'\in \c{T}$ is \textbf{minimal} if there exists no $T''\in \c{T}$ with $V(T'')\subsetneq V(T')$.
	
	\begin{definition}
		Let $G$ be a graph and $T$ a tree with $t:=v(T)$.  Given integers $C'\le C''$, we say that a sequence $(\Phi_{t},\ldots,\Phi_1)$ is an $(a,b,C',C'',T)$-\textbf{refinement sequence} for a graph $G$ if the following conditions hold:
		\begin{enumerate}[label=(S\arabic*)]
			\item We have $\Phi_t\sub \cdots \sub \Phi_1\sub \Mon(T,G)$ and each $\Phi_i$ is distinguishing.\label{condition good sequence nested and distinguishing}
			\item There exist integers $C'\le C_{1}\le C_{2}\cdots \le C_{t-1}\le C''$ and a $T$-subtree system $\c{T}$ such that $\c{T}_{\phi,\Phi_i,C_i}=\c{T}$ for all $\phi\in \Phi_{i+1}$.\label{condition good sequence T=T}
			\item For each $1\leq i\leq t-1$, every $\phi\in \Phi_{i+1}$ and every minimal subtree $T'\in \c{T}_{\phi,\Phi_i,C_i}=\c{T}$,  there exist maps $\psi^{(1)},\ldots,\psi^{(C')}\in \Psi(T',\phi(N(T'));\Phi_1)$ whose images are pairwise disjoint.\label{condition good sequence maps are pairwise disjoint for minimal}
			\item For each $i>1$ and every $\phi\in \Phi_{i}$, there does not exist an $(a,b,\phi,\Phi_{i-1})$-pseudo-piercing set for $\mathcal{T}$.\label{condition good sequence not pseudo-pierced}
			\item We have $|\Phi_t|\ge (C'')^{-1} \mon(T,G)$.\label{condition good sequence Phi_t not too small}
		\end{enumerate}
		For such a sequence we say that $\c{T}$ is the associated subtree system and that $C_1,\ldots,C_{t-1}$ are the associated integers. 
	\end{definition}
	Crucially, we can prove that every graph either has such a refinement sequence or few copies of $T$.
	
	\begin{lemma}\label{good sequence or few copies}
		For every tree $T$ and integers $a,b,C'$, there exists an integer $C''$ such that every graph $G$ either has an  $(a,b,C',C'',T)$-refinement sequence or has at most $C''e(G)^{a}v(G)^b$ copies of $T$.
	\end{lemma}
	
	\begin{proof}
		If $\Mon(T,G)$ is empty then trivially $G$ has at most $C'' e(G)^a v(G)^b$ copies of $T$,  so we assume $\Mon(T,G)$ is non-empty. Let $\Phi'_1$ be the large family of distinguishing monomorphisms guaranteed from \Cref{distinguishing large set of maps} and let $C'_1:=C'$.  Iteratively given that we have defined $\Phi'_i,C'_i$, we define the following set of objects:
		\begin{itemize}
			\item Let $\Psi_i'\sub \Phi'_i$ be the set of maps $\phi$ such that $\c{T}_{\phi,\Phi_{i}',C_{i}'}$ has an $(a,b,\phi,\Phi'_{i})$-pseudo-piercing set and let $\Phi''_i=\Phi'_i\sm \Psi_i'$.
			\item For a subtree system $\c{T}$, let $\Phi''_{i,\c{T}}\sub \Phi''_i$ denote the set of $\phi$ which have $\c{T}_{\phi,\Phi_i',C_i'}=\c{T}$.
			\item By the pigeonhole principle, there exists some $\c{T}$ such that $|\Phi''_{i,\c{T}}|\ge 2^{-2^{v(T)}}|\Phi''_i|$.  We let $\c{T}_i$ be any such family and define $\Phi'_{i+1}:=\Phi''_{i,\c{T}_i}$, and $C'_{i+1}:=\max_{T'\subseteq T}\Pow(T',C'_i)$, where the maximum goes over all subtrees of $T$, and $\Pow(T',C'_i)$ is the constant guaranteed by \Cref{lemma:sunflower}.
		\end{itemize}

		We observe the following.
		
		\begin{claim}\label{claim the the subtrees are nested}
			If $\Phi''_{i+1}\ne \emptyset$ then $\c{T}_{i+1}\sub \c{T}_i$.
		\end{claim}
		
		\begin{proof}
			By definition, we have
			\begin{equation}\label{equation containment of Phi i+1 and Phi i}
				\Phi''_{i+1,\mathcal{T}_{i+1}}\subseteq \Phi''_{i+1}\subseteq \Phi'_{i+1}=\Phi''_{i,\mathcal{T}_i}.
			\end{equation}
			By assumption we have $\Phi''_{i+1}\neq \emptyset$, which further implies that $\Phi''_{i+1,\mathcal{T}_{i+1}}\neq \emptyset$, so there exists some $\phi^*\in \Phi''_{i+1,\mathcal{T}_{i+1}}$. By~\eqref{equation containment of Phi i+1 and Phi i}, we also have $\phi^*\in \Phi''_{i,\mathcal{T}_i}$. Thus, by definition, $\c{T}_{i+1}=\c{T}_{\phi^*,\Phi'_{i+1},C'_{i+1}}$ and $\mathcal{T}_i=\c{T}_{\phi^*,\Phi'_i,C'_i}$. 
			
			Observe that for any arbitrary sets of monomorphisms $\Psi'\sub \Psi$ and integers $p'\ge p$, we have $\c{T}_{\psi,\Psi',p'}\sub \c{T}_{\psi,\Psi,p}$ for any $\psi\in\Psi'$ simply by definition of these subtree systems. Using this, we note that since $\Phi'_{i+1}\subseteq \Phi'_i$ and $C'_{i+1}\geq C'_i$, we have that $\c{T}_{\phi^*,\Phi'_{i+1},C'_{i+1}}\sub \c{T}_{\phi^*,\Phi'_i,C'_i}$. Putting everything together, we have
			\[
			\c{T}_{i+1}=\c{T}_{\phi^*,\Phi'_{i+1},C'_{i+1}}\sub \c{T}_{\phi^*,\Phi'_i,C'_i}=\c{T}_i.
			\]
			
		\end{proof}
		
		We also seek to understand the sizes of the $\Phi'_i$ sets.
		
		\begin{claim}\label{claim Phi does not get too small too fast}
			If $i$ is such that $|\Phi'_j|\ge 2 |\Psi_j'|$ for all $1\le j\le i$, then $|\Phi_{i+1}'|\ge 2^{-i-i2^{v(T)}} |\Phi'_1|$. 
		\end{claim}
		
		\begin{proof}
			Note that if $|\Phi'_j|\ge 2 |\Psi_j'|$ then $|\Phi''_j|\geq 2^{-1}|\Phi'_j|$, and so
			\[
			|\Phi_{j+1}'|=|\Phi''_{j,\c{T}_j}|\ge 2^{-2^{v(T)}} |\Phi''_j|\ge 2^{-2^{v(T)}}\cdot 2^{-1} |\Phi'_j|.
			\]
			The stated bound then follows by induction.
		\end{proof}
		
		Now if there's some $1\le i \le v(T)2^{2^{v(T)}}$ with $|\Phi'_i|<2|\Psi_i'|$, say $i'$ is the smallest such $i$, then because $|\Psi_{i'}|=O(e(G)^{a}v(G)^b)$ by \Cref{piercing equals few copies} and 
		\[
		|\Phi'_{i'}|\ge 2^{-(i'-1)-(i'-1)2^{v(T)}} |\Phi'_1|\geq 2^{-v(T)2^{2^{v(T)}}(1+2^{v(T)})}|\Phi'_1|
		\]
		by the minimality of $i'$ and \Cref{claim Phi does not get too small too fast}; and because $|\Phi_1'|=\Omega(\mon(T,G))$ by construction, we in total conclude that 
		\[
		\mon(T,G)=O(|\Phi_1'|)=O(|\Phi'_{i'}|)=O(|\Psi'_{i'}|)=O(e(G)^{a}v(G)^b).
		\]
		Let $C^*$ be an integer large enough that $\mathrm{mon}(T,G)\leq C^*e(G)^{a}v(G)^b$ in this case. Then as long as we choose $C''\geq C^*$, $G$ has few copies of $T$. As such, we may assume that no such $i'$ exists.
		
		With the above and Claim~\ref{claim Phi does not get too small too fast}, have that 
		\begin{equation}\label{equation Phi_i' is not too small}
			|\Phi'_i|\geq 2^{-(i-1)-(i-1)2^{v(T)}}|\Phi'_1|=\Omega(\mon(T,G))
		\end{equation}
		for all $1\le i\le v(T) 2^{2^{v(T)}}$.  Since all of these sets are in particular non-empty, by \Cref{claim the the subtrees are nested} and the pigeonhole principle there exists some $i^*<v(T)2^{2^{v(T)}}$ and $T$-subtree system $\c{T}$ such that $\c{T}_{i^*}=\c{T}_{i^*+1}\cdots=\c{T}_{i^*+v(T)-1}=\c{T}$.  We aim to prove the result with $\Phi_j:=\Phi'_{i^*+j-1}$, $C_j:=C'_{i^*+j-1}$ for $1\leq j\leq v(T)$. Towards this, let $C^{**}$ be the implicit constant in~\eqref{equation Phi_i' is not too small} applied with $i:=v(T) 2^{2^{v(T)}}$, and set $C'':=\max\left\{C^*,1/C^{**},C_{v(T) 2^{2^{v(T)}}}'\right\}$.
		
		\ref{condition good sequence nested and distinguishing} follows from the construction of the $\Phi'_i$'s, and the fact that $\Phi_1'$ is distinguishing. Similarly, \ref{condition good sequence T=T} follows from construction and the choices of the $\Phi_i$'s, while \ref{condition good sequence not pseudo-pierced} follows from the choice of the $\Psi_i'$'s. \ref{condition good sequence Phi_t not too small} follows since $|\Phi_{v(T)}|=|\Phi'_{i^*+v(T)-1}|\geq C^{**}\mon(T,G)\geq (C'')^{-1}\mon(T,G)$. Thus, let us focus on \ref{condition good sequence maps are pairwise disjoint for minimal}.
		
		Fix $i$ with $1\leq i\leq v(T)-1$, let $\phi\in \Phi_{i+1}$ and $T'\in \c{T}_{\phi,\Phi_i,C_i}$ a minimal element. Since $\c{T}_{\phi,\Phi_i,C_i}=\mathcal{T}=\c{T}_{\phi,\Phi_{i+1},C_{i+1}}$, we also have $T'\in \c{T}_{\phi,\Phi_{i+1},C_{i+1}}$, which by definition means that $|\Psi(T',\phi(N(T'));\Phi_{i+1})|\geq C_{i+1}\ge \Pow(T',C_i)$.  By \Cref{lemma:sunflower}, there exists a set $R\subsetneq V(T')$ and distinct maps $\psi^{(1)},\ldots,\psi^{(C_i)}\in \Psi(T',\phi(N(T'));\Phi_{i+1})$ which all agree on $R$ and whose images are pairwise disjoint outside of $R$. If $R=\emptyset$, then the images of the $\psi_i$'s are pairwise disjoint, and since $\Psi(T',\phi(N(T'));\Phi_{i+1})\subseteq \Psi(T',\phi(N(T'));\Phi_{1})$,~\ref{condition good sequence maps are pairwise disjoint for minimal} is satisfied. Thus, let us assume $R\neq \emptyset$. 
		
		By definition of $\Psi(T',\phi(N(T'));\Phi_{i+1})$, for each $1\le p\le C_i$, there exists $\phi^{(p)}\in \Phi_{i+1}$ such that $\phi^{(p)}|_{V(T')}=\psi^{(p)}$ and such that $\phi^{(p)}(N(T'))=\phi(N(T'))$. Let $T''$ denote an arbitrary connected component of $T'-R$, and note that $N(T'')\sub N(T')\cup R$, and as such $\phi^{(p)}(N(T''))=\phi^{(p')}(N(T''))$ for all $p,p'$. This gives us that $\phi^{(p)}|_{V(T'')}\in \Psi(T'',\phi^{(1)}(N(T''));\Phi_{i+1})\subseteq \Psi(T'',\phi^{(1)}(N(T''));\Phi_i)$, and in fact the images of the $\phi^{(p)}|_{V(T'')}$'s are pairwise disjoint by our choice of $T''$, and thus $|\Psi(T'',\phi_1(N(T''));\Phi_i)|\geq C_i$.  This means $T''\in \mathcal{T}_{\phi^{(1)},\Phi_i,C_i}=\mathcal{T}$, a contradiction to $T'$ being a minimal element. We conclude the result.
	\end{proof}
	
	The result above holds for arbitrary trees $T$ and may be useful to further refinements of \Cref{k-1 edges for trees}.  For the specific case of paths, we can derive a convenient corollary of \Cref{good sequence or few copies} in terms of another technical condition.  For this, here and throughout we assume the path graph $P_t$ to have $V(P_t)=[t]$, meaning that subtrees of $P_t$ all have vertex sets which are some interval $[\ell,r]$ for some choice of $\ell,r$.
	
	\begin{definition}
		Let $\mathcal{T}$ be a $P_t$-subtree system, let $T_1,T_2,\dots,T_{2k}\in\mathcal{T}$, and let $\widehat{M}\subseteq E(P_t)$ be a matching of size $k$ in $P$, say with $\widehat{M}=\{\hat{x}_1\hat{x}_2,\hat{x}_3\hat{x}_4,\dots,\hat{x}_{2k-1}\hat{x}_{2k}\}$ where $\hat{x}_i<\hat{x}_j$ whenever $i<j$. For each $i\in [2k]$, let $T_i=[\ell_i,r_i]$. We say the tuple $(T_1,T_2,\dots,T_{2k})$ is a \textbf{$(k,t,\mathcal{T})$-nice tuple induced by $\widehat{M}$} if the following holds.
		\begin{enumerate}[label=(N\arabic*)]
			\item $\hat{x}_i\in V(T_i)\sm \bigcup_{j\neq i} V(T_j)$ for all $i\in [2k]$,\label{nice SDR}
			\item Each $T_i$ is minimal in $\mathcal{T}$. \label{nice minimal}
		\end{enumerate}
		We call  $\widehat{M}$ the matching associated with the tuple $(T_1,T_2,\dots,T_{2k})$. We say a $(k-1,1,C',C'',P_t)$-refinement sequence $(\Phi_t,\dots,\Phi_1)$ with associated subtree system $\mathcal{T}$ is \textbf{$k$-nice} if there exists a matching $\widehat{M}\subseteq E(P_t)$ of size $k$ and a $(k,t,\mathcal{T})$-nice tuple induced by $\widehat{M}$.
	\end{definition}
	
	It is worth noting that given a $P_t$-subtree system and a matching $\widehat{M}$ of size $k$, if a $(k,t,\mathcal{T})$-nice tuple induced by $\widehat{M}$ exists, it is unique. This will not be necessary to our application of the definition though, so we omit a proof. From this point forward, if we fix a $(k,t,\mathcal{T})$-nice tuple $(T_1,T_2,\dots,T_{2k})$, we will always let $\ell_i,r_i\in [2k]$ be integers such that $T_{i}=[\ell_i,r_i]$.
	
	\begin{lemma}\label{lemma few copies implies k-nice refinement sequence}
		For all integers $q,k,t$, there exists an integer $C_0$ such that for all $C'\geq C_0$, and $C''$ as guaranteed by \Cref{good sequence or few copies} at least one of the following holds for every graph $G$:
		\begin{itemize}
			\item $G$ contains at most $O(e(G)^{k-1} v(G))$ copies of $P_t$.
			\item $G$ contains a subgraph isomorphic to an element of $\c{F}_{P_t,k+1}^q$.
			\item $G$ contains a $k$-nice $(k-1,1,C',C'',P_t)$-refinement sequence. 
		\end{itemize}
	\end{lemma}
	
	\begin{proof}
		We will prove this result with $C_0$ equal to the constant from \Cref{GT Pierces by Edges}. Let $C'\geq C_0$, and let us assume $G$ is $\c{F}_{P_t,k+1}^q$-free. By \Cref{good sequence or few copies}, either $G$ has $O(e(G)^{k-1}v(G))$ copies of $P_t$, in which case we are done, or $G$ has a $(k-1,1,C',C'',P_t)$-refinement sequence, say $(\Phi_{t},\ldots,\Phi_1)$ with associated subtree system $\c{T}$ and associated integers $C_1,\dots,C_{t-1}$. If $\mathcal{T}$ has a $(k-1,1)$-piercing set, then for every $\phi\in \Phi_t$, this $(k-1,1)$-piercing set is also a $(k-1,1,\phi,\Phi_{t-1})$-pseudo-piercing set for $\mathcal{T}_{\phi,\Phi_{t-1},C_{t-1}}$.  By \Cref{piercing equals few copies} applied with $\Phi=\Phi_{t-1}$ and $ \Phi'=\Phi_t$, we have $|\Phi_t|=O(e(G)^{k-1} v(G))$.  This along with \ref{condition good sequence Phi_t not too small} gives  $\mon(P_t,G)=O(e(G)^{k-1} v(G))$ and we are done. Thus, let us assume that no such $(k-1,1)$-piercing set exists.
		
		Because $\c{T}=\c{T}_{\phi,\Phi_{t-1},C_{t-1}}$ and $C_{t-1}\ge C'\ge C_0$, we have from \Cref{GT Pierces by Edges} that $\c{T}$ can be pierced by some set $\widehat{M}$ of $m\le k$ edges, say 
		\[
		\widehat{M}=\{\hat{x}_1\hat{x}_2,\hat{x}_3\hat{x}_4,\ldots,\hat{x}_{2m-1}\hat{x}_{2m}\},
		\]
		where without loss of generality we may assume $\hat{x}_1\le \hat{x}_2\le \cdots \le \hat{x}_{2m}$. In fact, it is not difficult to see that we must have $m=k$ and $\hat{x}_1<\hat{x}_2<\dots<\hat{x}_{2k}$ (so $\widehat{M}$ is a matching with $k$ edges) since otherwise $\mathcal{T}$ could be $(k-1,1)$-pierced. For each $i\in [2k]$, let $\hat{y}_i$ denote the vertex matched to $\hat{x}_i$ by $\widehat{M}$ (so $\hat{y}_i=\hat{x}_{i-1}$ if $i$ is even, and $\hat{y}_i=\hat{x}_{i+1}$ if $i$ is odd).
		
		\begin{claim}
			For each $i\in [2k]$, there exists at least one subtree $T'\in\mathcal{T}$ such that $T'$ contains $\hat{x}_i$ and does not contain $\hat{x}_j$ for any $j\neq i$.
		\end{claim}
		
		\begin{proof}
			If there is no such tree $T'$, then $\widehat{M}\setminus \{\hat{x}_i\hat{y}_i\}$ along with the vertex $\hat{y}_i$ constitutes a $(k-1,1)$-piercing set for $\mathcal{T}$, which we assumed does not exist.
		\end{proof}
		
		Now, for each $i\in [2k]$, let $T_i$ be a minimal tree in $\mathcal{T}$ such that $T_i$ contains $\hat{x}_i$ and does not contain $\hat{x}_j$ for any $j\neq i$. Then $(T_1,T_2,\dots,T_{2k})$ is a $(k,t,\mathcal{T})$-nice tuple induced by $\widehat{M}$ by design. Thus, $(\Phi_t,\dots,\Phi_1)$ is a $k$-nice $(k-1,1,C',C'',P_t)$-refinement sequence.
	\end{proof}
	
	For the rest of the proof we aim to show that nice tuples imply the existence of subgraphs from $\c{F}_{P_t,k+1}^q$ through some case analysis. First, let us gather a few properties of $(k,t,\mathcal{T})$-nice tuples. 
	
	\begin{observation}\label{observation properties of nice tuples}
		Let $\mathcal{T}$ be a $P_t$-subtree system and let $(T_1,T_2,\dots,T_{2k})$ be a $(k,t,\mathcal{T})$-nice tuple with associated matching $\widehat{M}=\{\hat{x}_1\hat{x}_2,\dots,\hat{x}_{2k-1}\hat{x}_{2k}\}$, where $\hat{x}_1<\hat{x}_2<\dots<\hat{x}_{2k}$. Then all the following hold.
		\begin{enumerate}[label=(O\arabic*)]
			\item $\ell_1<\ell_2<\dots<\ell_{2k}$ and $r_1<r_2<\dots<r_{2k}$.\label{nice observation trees in order}
			\item If $i<j$ then $\hat{x}_i<\ell_j\leq r_j$ and $\ell_i\leq r_i<\hat{x}_j$.\label{nice observation order of x and l and r}
			\item If $i$ is odd, then $r_i=\hat{x}_i$ and if $i$ is even, $\ell_i=\hat{x}_i$. In particular, if $i$ is odd, $r_i+1=\ell_{i+1}$.\label{nice observation endpoints are x}
			\item If $i<j-1$, then $r_i<\ell_j-1$.\label{nice observation far apart}
			\item If $i<j$ and if $V(T_i)\cap V(T_j)\neq \emptyset$, then $j=i+1$ and $i$ is even.\label{nice observation trees intersection}
		\end{enumerate}
	\end{observation}
	
	\begin{proof}
		Let $i<j$. Then by \ref{nice SDR}, we have $\ell_i\leq \hat{x}_i<\hat{x}_j\leq r_j$. If we had $\ell_j\leq \hat{x}_i$, this would give us  that $\hat{x}_i\in V(T_j)$, contradicting \ref{nice SDR}, so $\ell_i\leq \hat{x}_i<\ell_j\leq r_j$. By a symmetric argument, we can get that $\ell_i\leq r_i<\hat{x}_j\leq r_j$. These together establish both \ref{nice observation trees in order} and \ref{nice observation order of x and l and r}.
		
		Now, if $i$ is odd, note that $\hat{x}_i\hat{x}_{i+1}\in \widehat{M}$, so in particular $\hat{x}_{i+1}=\hat{x}_i+1$. Then by \ref{nice SDR} $T_i$ contains $\hat{x}_i$ but not $\hat{x}_{i+1}$, while $T_{i+1}$ contains $\hat{x}_{i+1}$ but not $\hat{x}_i$, so we must have $r_i=\hat{x}_i$ and $\ell_{i+1}=\hat{x}_{i+1}$, giving \ref{nice observation endpoints are x}.
		
		Now assume $i<j-1$. then by \ref{nice observation order of x and l and r}, $r_i<\hat{x}_{i+1}\leq \hat{x}_{j-1}<\ell_{j}$, so indeed $r_i<\ell_j-1$, giving \ref{nice observation far apart}.
		
		Finally, if $i<j$, and $V(T_i)\cap V(T_j)\neq \emptyset$, then $r_i\geq \ell_j$, so by \ref{nice observation far apart}, we must have $j=i+1$, and by \ref{nice observation endpoints are x} we could not have $r_i\geq \ell_{i+1}$ if $i$ was odd, so $i$ is even, giving \ref{nice observation trees intersection}.
	\end{proof}
	
	In what follows, to avoid cumbersome language, we may use some of the properties listed above without evoking a reference to \Cref{observation monomorphisms same as copies} unless we believe such a reference is necessary for clarity.
	
	\subsection{Using \texorpdfstring{$k$}{k}-nice tuples}
	
	We wish to show that $k$-nice refinement sequences imply that $G$ contains an element of $\mathcal{F}^q_{P_t,k+1}$.  In each case, the element of $\mathcal{F}^q_{P_t,k+1}$ we find will be close to a generalized theta graph $\theta_{a,b,c}$, which we recall is the graph obtained by taking $a$ paths on $1+cb$ vertices which intersect at the vertices $1,1+b,\ldots,1+cb$.  In particular, our first case relies on the following.
	
	\begin{lemma}\label{theta implies F}
		Given integers $q,t,k\ge 2$, the graph $\theta_{q,b,t}$ contains an element of $\c{F}_{P_t,k+1}^q$ as a subgraph whenever $2\le b \le \frac{t-3}{k-1}$.
	\end{lemma}
	\begin{proof}
		Let $v_i^{(p)}$ denote the $i$th vertex in the $p$th path making up $\theta_{q,b,t}$, and for $i$ of the form $i=1+rb$ we additionally write $v_i$ to denote the vertex which is the $i$th vertex of every path.  For each $1\le p\le q$, let $P^{(p)}$ denote the path on the vertices $v_b^{(p)},v_{b+1}^{(p)},\ldots, v_{t+b-1}^{(p)}$, noting that this path is well defined since each path in $\theta_{q,b,t}$ goes up to index $1+tb\ge t+b-1$.
		
		We claim that the paths $P^{(1)},\ldots,P^{(q)}$ define an element of $\c{F}_{P_t,k+1}^q$.  Indeed, these paths all have $t$ vertices and all agree on the set of vertices $R=\{v_{1+rb}: 1+rb\le t+b-1\}$ and are otherwise pairwise disjoint, so it suffices to show that $P^{(p)}-R$ has at least $k+1$ connected components for each $p$.  And indeed, the vertices $v_b^{(p)},v_{2b}^{(p)},\ldots,v_{kb}^{(p)},v_{2+kb}^{(p)}$ lie in distinct components of $P^{(p)}-R$, where here we implicitly used that $v_{2+kb}^{(p)}\in P^{(p)}$, i.e.\ that $2+kb\le t+b-1$ by hypothesis, as well as that each of these vertices do not lie in $R$ since $b\ge 2$.  This gives the desired result.
	\end{proof}
	
	This in turn yields the following technical result which will be key to our analysis.
	
	\begin{lemma}\label{nice tuples can zigzag}
		Consider integers $q$, $t$, $C'$ and $C''$ with $C''\geq C'\geq 1+q+(t-1)(1+q(\frac{t-3}{k-1}-1))$, and let $G$ be a graph which contains at least one copy of $P_t$ and a $k$-nice $(k-1,1,C',C'',P_t)$-refinement sequence $(\Phi_t,\dots,\Phi_1)$ with $(k,t,\mathcal{T})$-nice tuple $(T_1,T_2,\dots,T_{2k})$. If there exists some $i^*$ with $1<i^*<2k$ such that $|V(T_{i^*})|\leq\frac{t-3}{k-1}-1$, and such that $r_{i^*}+1\in V(T_{{i^*}+1})$ and $\ell_{i^*}-1\in V(T_{{i^*}-1})$, then $G$ contains an element of $\mathcal{F}^q_{P_t,k+1}$ as a subgraph.
	\end{lemma}
	\begin{proof}
		Set $b:=|V(T_{i^*})|+1$. Our goal will be to find a copy of a generalized theta graph $\theta_{q,b,t}$ in $G$, from which the result will follow from \Cref{theta implies F}.
		
		By symmetry, we may assume that ${i^*}$ is odd.  Let $\widehat{M}=\{\hat{x}_1\hat{x}_2,\dots,\hat{x}_{2k-1}\hat{x}_{2k}\}$ be the matching in $P_t$ associated with $(T_1,T_2,\dots,T_{2k})$ where we assume $\hat{x}_1<\hat{x}_2<\dots<\hat{x}_{2k}$.
		
		Given $j>1$ and a map $\phi\in \Phi_j$, let
		\[
		\mathcal{L}_{\phi,j}:=\{\phi'(\ell_{i^*}-1)\mid \phi'\in \Phi_{j-1},\ \phi'(r_{i^*}+1)=\phi(r_{i^*}+1)\},
		\]
		and
		\[
		\mathcal{R}_{\phi,j}:=\{\phi'(r_{i^*}+1)\mid \phi'\in \Phi_{j-1},\ \phi'(\ell_{i^*}-1)=\phi(\ell_{i^*}-1)\}.
		\]
		\begin{claim}\label{claim XL and XR are big}
			If $j>1$ and $\phi\in \Phi_j$, then $|\mathcal{L}_{\phi,j}|\geq t$ and $|\mathcal{R}_{\phi,j}|\geq t$.
		\end{claim}
		We emphasize that the proof of this claim is the only place in our argument where we rely on the more general notion of pseudo-piercing.
		\begin{proof}
			Assume to the contrary that one of these sets is small.
			
			\textbf{Case 1:} $|\mathcal{L}_{\phi,j}|<t$. In this case we set $\widehat{M}':=\widehat{M}\setminus \{r_{i^*-2}\ell_{i^*-1}\}$, and we aim to show $\widehat{M}'$ along with the vertex $r_{i^*-2}$ is a $(k-1,1,\phi,\Phi_{j-1})$-pseudo-piercing set for $\mathcal{T}$.
			
			Let $\widehat{W}:=V(\widehat{M}')\cup\{r_{i^*-2}\}$. Since $V(\widehat{M})$ pierces $\mathcal{T}$, if any tree $T'$ is not pierced by $\widehat{W}$, then $T'$ must contain $\ell_{i^*-1}$ and no vertices of $\widehat{W}=V(\widehat{M})\setminus \{\ell_{i^*-1}\}$. In particular $T'$ does not contain $r_{i^*-2}=\ell_{i^*-1}-1$ since $i^*$ is odd, so $T'$ must be of the form $[\ell_{i^*-1},r]$ for some $r$. Since $T_{i^*-1}$ is the minimal tree of this form in $\mathcal{T}$ by definition of nice sequences, we must have  $V(T_{i^*-1})\subseteq V(T')$. In particular, the hypothesis of the lemma implies $\ell_{i^*}-1\in V(T_{i^*-1})\sub V(T')$. But then since $r_{i^*}+1\in \widehat{W}$ (since $i^*$ is odd, one of the edges in $\widehat{M}$ is $\{r_{i^*}, r_{i^*}+1\}$), we have that
			\[
			\left\{\phi'(\ell_{i^*}-1)\mid\phi'\in \Phi_{j-1},\ \phi'(\widehat{W})=\phi(\widehat{W})\right\}\subseteq \mathcal{L}_{\phi,j}.
			\]
			Thus since $|\mathcal{L}_{\phi,j}|<t$, $T'$ is $(\phi,\Phi_{j-1})$-pseudo-pierced by $\widehat{W}$, so $\mathcal{T}$ is $(k-1,1,\phi,\Phi_{j-1})$-pseudo-pierced, contradicting \ref{condition good sequence not pseudo-pierced}.
			
			\textbf{Case 2:} $|\mathcal{R}_{\phi,j}|<t$. In this case, we aim to show that the the edges of \[\widehat{M}':=(\widehat{M}\setminus \{r_{i^*}\ell_{i^*+1},r_{i^*-2}\ell_{i^*-1}\})\cup \{(\ell_{i^*}-1)\ell_{i^*}\},\] along with the vertex $r_{i^*-2}$ give a $(k-1,1,\phi,\Phi_{j-1})$-pseudo-piercing set. 
			
			Let $T'\in\mathcal{T}$ be a tree that is not pierced by $\widehat{W}:=V(\widehat{M}')\cup \{r_{i^*-2}\}$. Since $V(\widehat{M})$ pierces $\mathcal{T}$, we must have that $T'$ contains a vertex from $\{r_{i^*},\ell_{i^*+1},\ell_{i^*-1}\}$.
			
			\textbf{Subcase 2.1:} $T'$ contains $r_{i^*}$ but not $\ell_{i^*+1}$. Then by the ordering of vertices in $P_t$, $T'$ must be of the form $[\ell,r_{i^*}]$ for some $\ell\leq r_{i^*}$, but $T_{i^*}$ is the minimal tree in $\mathcal{T}$ of this form, so we must have $V(T_{i^*})\subseteq V(T')$.  In particular, we have that $\ell_{i^*}\in V(T')$, so $\widehat{W}$ does pierce $T'$.
			
			\textbf{Subcase 2.2:} $T$ contains $\ell_{i^*-1}$. Then we may also assume $T'$ does not contain $r_{i^*-2}\in \widehat{W}$, so $T'$ is of the form $[\ell_{i^*-1},r]$ for some $r$, but $T_{i^*-1}$ is the minimal tree of this form so we must have $V(T_{i^*-1})\subseteq V(T')$.  In particular $\ell_{i^*}-1\in V(T')$, so $\widehat{W}$ does pierce $T$.
			
			\textbf{Subcase 2.3:} $T'$ contains $\ell_{i^*+1}$. In this case, we claim that $\widehat{W}$ is a set of vertices which $(\phi,\Phi_{j-1})$-pseudo-pierces $T'$. Indeed, since $\ell_{i^*}-1\in \widehat{W}$,
			\[
			\left\{\phi'(r_{i^*}+1)\mid\phi'\in \Phi_{j-1},\ \phi'(\widehat{W})=\phi(\widehat{W})\right\}\subseteq \mathcal{R}_{\phi,j}.
			\]
			so since $|\mathcal{R}_{\phi,j}|<t$, $T'$ is pseudo-pierced.
			
			Thus $\widehat{W}$ is indeed a $(k-1,1,\phi,\Phi_{j-1})$-pseudo-piercing set for $\mathcal{T}$, which contradicts \ref{condition good sequence not pseudo-pierced}.
		\end{proof}
		
		With this in hand, we can start building the copy of $\theta_{q,b,t}$ in $G$. We will build this recursively. To initalize, let $\phi_t\in \Phi_t$ be any map, noting by \ref{condition good sequence Phi_t not too small} and the fact that $G$ contains a copy of $P_t$ that $\Phi_t\neq \emptyset$.  Set $u_t:=\phi_t(\ell_{i^*}-1)$ and $v_t:=\phi_t(r_{i^*}+1)$. Since $q\leq C'$ and $T_{i^*}$ is minimal in $\mathcal{T}$, by \ref{condition good sequence maps are pairwise disjoint for minimal}, there exist maps $\psi_{t}^{(1)},\psi_{t}^{(2)},\dots,\psi_{t}^{(q)}\in \Psi(T_{i^*},\{u_t,v_t\},\Phi_1)$ (noting that $\{u_t,v_t\}=\phi_t(N(T^*))$) whose images are pairwise disjoint. The graph-images of these paths along with the vertices $u_t$ and $v_t$ give us a copy of the theta graph $\theta_{q,b}$ in $G$ which we will denote by $\theta_t$
		
		Now, for some $j$ with $1\leq j<t$, assume that for all $i$ with $j<i\leq t$ we have already chosen maps $\phi_i\in \Phi_{i}$ and defined $u_i:=\phi_i(\ell_{i^*}-1)$ and $v_i:=\phi_i(r_{i^*}+1)$, and we have defined maps $\psi_i^{(1)},\dots,\psi_i^{(q)}\in \Psi(T_{i^*},\{x_i,y_i\},\Phi_1)$ each of which constitutes a copy $\theta_i$ of $\theta_{q,b}$ in $G$ where for any $i<i'$, 
		\[
		V(\theta_i)\cap V(\theta_{i'})=\begin{cases}
			\emptyset&\text{ if }i'-i\geq 2,\\
			\{u_i\}=\{u_{i'}\}&\text{ if }i'=i+1\text{ and }i\text{ is even,}\\
			\{v_i\}=\{v_{i'}\}&\text{ if }i'=i+1\text{ and }i\text{ is odd.}
		\end{cases}
		\]
		We will now show how to proceed when $j$ is even, the case where $j$ is odd will be symmetric. 
		
		By \Cref{claim XL and XR are big}, we have $|\mathcal{L}_{\phi_{j+1},j+1}|\geq t$ and $|\mathcal{R}_{\phi_{j+1},j+1}|\geq t$. Let $v\in \mathcal{R}_{\phi_{j+1},j+1}\setminus \{v_i\mid j<i\leq t\}$, and let $\phi_j\in \Phi_j$ be a map such that $u_j:=\phi_j(\ell_{i^*-1})=\phi_{j+1}(\ell_{i^*-1})=u_{j+1}$ and such that $v_j:=\phi_j(r_{i^*+1})=v$ (such a map exists by the definition of $\mathcal{R}_{\phi_{j+1},j+1}$), noting that $v_j$ does not intersect $\bigcup_{i=j+1}^tV(\theta_i)$ since the $\Phi_i$'s are distinguishing and because $v=v_j$ was chosen to avoid $v_i$ for $i>j$.
		
		By \ref{condition good sequence maps are pairwise disjoint for minimal} we can choose maps $\psi^{(1)}_j,\dots\psi^{(C')}_{j}\in \Psi(T_{i^*},\{u_j,v_j\},\Phi_1)$ with pairwise-disjoint images. At most $\left|\bigcup_{i=j+1}^tV(\theta_i)\right|=1+(t-j)(1+q(b-1))$ of these maps have images that intersect some $\theta_i$ with $i>j$, and since  $C'\geq q+1+(t-j)(1+q(b-1))$ this means at least $q$ of these maps avoid the $\theta_i$'s, assume without loss of generality that $\psi_j^{(1)},\dots,\psi_j^{(q)}$ are $q$ maps whose images do not intersect $\bigcup_{i=j+1}^tV(\theta_i)$. The graph-images of these maps along with the vertices $u_i$ and $v_i$ constitute a copy of $\theta_{q,b}$ in $G$, call it $\theta_j$, and note that $\theta_j$ indeed intersects $\theta_{j+1}$ only in $u_j=u_{j+1}$, and does not intersect any other $\theta_i$ for $i>j$.
		
		We may continue this process until we have built $\theta_t,\theta_{t-1},\dots,\theta_1$, which intersect in such a way to give a copy of $\theta_{q,b,t}$ as desired.  By hypothesis we have $b=|V(T_{i^*})|+1\le \frac{t-3}{k-1}$ and $b\ge 2$, so by \Cref{theta implies F} this $\theta_{q,b,t}\sub G$ contains an element of $\c{F}_{P_t,k+1}^q$ as a subgraph, proving the result.
	\end{proof}

	The remainder of this subsection is dedicated to showing that a $k$-nice refinement sequence with the properties needed to apply \Cref{nice tuples can zigzag} exists. In particular, we want to find a ``small'' tree $T_{i^*}$ such that $\ell_{i^*}-1\in V(T_{i^*-1})$ and $r_{i^*}+1\in V(T_{i^*+1})$. Before we can find such a tree, we need the following, which shows that $T_1$ and $T_{2k}$ must contain only the vertices $1$ and $t$ respectively.

	\begin{lemma}\label{ell2 equals 2}
		If $G$ contains a $k$-nice $(k-1,1,C',C'',P_t)$-refinement sequence $(\Phi_t,\dots,\Phi_1)$ with $(k,t,\mathcal{T})$-nice tuple $(T_1,T_2,\dots,T_{2k})$ such that $C'\ge 2q+t$  and such that $\ell_2\neq 2$ or $r_{2k-1}\neq t-1$, then $G$ contains a subgraph isomorphic to an element of $\c{F}_{P_t,k+1}^q$.
	\end{lemma}
	
	\begin{proof}
		By the symmetry of the problem it suffices to prove the conclusion under the assumption $r_{2k-1}\ne t-1$. Consider any $\phi\in \Phi_t$.  By \ref{condition good sequence maps are pairwise disjoint for minimal} together with the fact that each $T_i$ is a minimal element of $\c{T}=\c{T}_{\phi,\Phi_{t-1},C_{t-1}}$ and that $C'\ge 2q+t$; for each $i$ there exist maps $\psi_i^{(1)},\ldots,\psi_{i}^{(2q+t)}\in \Psi(T_i,\phi(N(T_i));\Phi_{t-1})$ with disjoint images. At most $t$ of these maps have images which intersect with the image of $\phi$, so we may assume without loss of generality that $\psi_i^{(1)},\ldots,\psi_{i}^{(2q)}\in \Psi(T_i,\phi(N(T_i));\Phi_{t-1})$ have images which are disjoint from the image of $\phi$. Let us define $I:=\bigcup_{s=1}^{k}V(T_{2s-1})$. By \Cref{observation properties of nice tuples}, we have that the sets $V(T_{2s-1})$ and $V(T_{2s'-1})$ are disjoint for each $s\neq s'$, $s,s'\in [k]$, which will be crucial to make sure the maps we introduce below are well-defined. We break our proof into two cases based on the value of $r_{2k-1}$.  
		
		\textbf{Case 1:} $r_{2k-1}=t-2$. For each integer $1\le p\le q$, we define a map $\phi^{(p)}:V(P_t)\to V(G)$ as follows:
		\begin{itemize}
			\item For each integer $1\le d\le t-1$ with $d\notin I$ we let $\phi^{(p)}(d)=\phi(d)$.
			\item For each integer $1\le d\le t-1$ with $d\in I$, say with $d\in V(T_{2s-1})$, we let $\phi^{(p)}(d)=\psi_{2s-1}^{(p)}(d)$.
			\item We let $\phi^{(p)}(t)=\psi_{2k-1}^{(q+p)}(t-2)$.
		\end{itemize}
		
		We emphasize this last step is well-defined since $r_{2k-1}=t-2$ implies $t-2\in [\ell_{2k-1},r_{2k-1}]=V(T_{2k-1})$ and hence lies in the domain of $\psi^{(q+p)}
		_{2k-1}$.
		
		\begin{figure}
			\begin{center}
				\begin{tikzpicture}
					\draw (-4,0)--(9.5,0);
					
					\draw (4.5,0)--(5.5,0.4)--(7,0.4)--(8,0);
					\draw (4.5,0)--(5.5,0.8)--(7,0.8)--(8,0);
					\draw (4.5,0)--(5.5,1.2)--(7,1.2)--(8,0);
					\draw (4.5,0)--(5.5,1.6)--(7,1.6)--(8,0);
					\draw (4.5,0)--(5.5,2)--(7,2)--(8,0);
					\draw (4.5,0)--(5.5,2.4)--(7,2.4)--(8,0);

					\draw (0,0)--(1,0.5)--(2.5,0.5)--(3.5,0);
					\draw (0,0)--(1,1)--(2.5,1)--(3.5,0);
					\draw (0,0)--(1,1.5)--(2.5,1.5)--(3.5,0);
					
					\draw (-4,0.5)--(-2.5,0.5)--(-1.5,0);
					\draw (-4,1)--(-2.5,1)--(-1.5,0);
					\draw (-4,1.5)--(-2.5,1.5)--(-1.5,0);

					\draw[very thick, color=blue] (-4,1)--(-2.5,1)--(-1.5,0)--(0,0)--(1,1)--(2.5,1)--(3.5,0)--(4.5,0)--(5.5,0.8)--(7,0.8)--(8,0)--(7,2);
					
					\draw[fill=black] (9,0) circle (2pt);
					
					\draw[fill=black] (8,0) circle (2pt);
					\draw[fill=black] (7,0) circle (2pt);
					\draw[fill=black] (5.5,0) circle (2pt);
					\draw[fill=black] (4.5,0) circle (2pt);
					\draw[fill=black] (3.5,0) circle (2pt);
					\draw[fill=black] (2.5,0) circle (2pt);
					\draw[fill=black] (1,0) circle (2pt);
					\draw[fill=black] (0,0) circle (2pt);
					\draw[fill=black] (-1.5,0) circle (2pt);
					\draw[fill=black] (-2.5,0) circle (2pt);
					
					\draw[fill=black] (7,1.6) circle (2pt);
					\draw[fill=black] (7,2) circle (2pt);
					\draw[fill=black] (7,2.4) circle (2pt);
					
					\draw[fill=black] (5.5,1.6) circle (2pt);
					\draw[fill=black] (5.5,2) circle (2pt);
					\draw[fill=black] (5.5,2.4) circle (2pt);
					
					\draw[fill=black] (7,0.4) circle (2pt);
					\draw[fill=black] (7,0.8) circle (2pt);
					\draw[fill=black] (7,1.2) circle (2pt);
					
					\draw[fill=black] (5.5,0.4) circle (2pt);
					\draw[fill=black] (5.5,0.8) circle (2pt);
					\draw[fill=black] (5.5,1.2) circle (2pt);
					
					\draw[fill=black] (2.5,0.5) circle (2pt);
					\draw[fill=black] (2.5,1) circle (2pt);
					\draw[fill=black] (2.5,1.5) circle (2pt);
					
					\draw[fill=black] (1,0.5) circle (2pt);
					\draw[fill=black] (1,1) circle (2pt);
					\draw[fill=black] (1,1.5) circle (2pt);
					
					\draw[fill=black] (-2.5,0.5) circle (2pt);
					\draw[fill=black] (-2.5,1) circle (2pt);
					\draw[fill=black] (-2.5,1.5) circle (2pt);
					
					\node at (6.25,2.8) {$T_{2k-1}$};
					\node at (1.75,2) {$T_{2k-3}$};
					\node at (-3.25,2) {$T_{1}$};
					\node at (5.5,-0.3) {$\ell_{2k-1}$};
					\node at (6.8,-0.4) {$r_{2k-1}$};
					\node at (8.4,0.3) {$t-1$};
					\node at (1,-0.3) {$\ell_{2k-3}$};
					\node at (2.5,-0.4) {$r_{2k-3}$};
					\node at (-2.5,-0.4) {$r_{1}$};
					\node at (-0.75,0.5) {\large{\textbf{\dots}}};
				\end{tikzpicture}
			\end{center}
			\caption{Finding a copy of $\mathcal{F}_{P_t,k+1}^q$ in Case 1 of \Cref{ell2 equals 2}. The graph-image of one $\phi^{(p)}$ is shown in blue.}\label{figure case 1 ell2 equals 2}
		\end{figure}
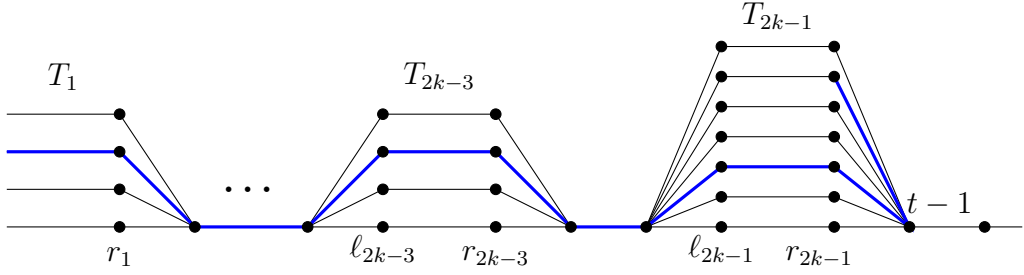
		
		Since $\Phi_{t-1}$ is distinguishing and by our choices of the $\psi_{i}^{(p)}$'s, the collection $\{\phi^{(p)}\mid p\in [q]\}$ consists of monomorphisms of $P_t$ into $G$ such that each monomorphism agrees on the image of the set $R:=V(P_t)\setminus (I\cup \{t\})$, and otherwise the image is pairwise disjoint from any other map in the collection. Furthermore, we note that $P_t-R$ will contain $k+1$ components; $k$ of them being $T_{2s-1}$ for $s\in [k]$, and one corresponding to the isolated vertex $\{t\}$. Thus, the graph-images of the $\phi^{(p)}$'s constitute an element of $\mathcal{F}^q_{P_t,k+1}$; see Figure~\ref{figure case 1 ell2 equals 2} for an illustration of this case.
		
		\textbf{Case 2:} $r_{2k-1}\ne t-2$. This implies $r_{2k-1}\le t-3$ since $r_{2k-1}\ne t-1$ by assumption and $r_{2k-1}<r_{2k}\le t$ by \Cref{observation properties of nice tuples}. Similar to the previous case, for each integer $1\le p\le q$ we define maps $\phi^{(p)}:V(P_t)\to V(G)$ as follows:
		\begin{itemize}
			\item For each integer $1\le d\le r_{2k-1}+1$ with $d\notin I$ we let $\phi^{(p)}(d)=\phi(d)$.
			\item For each integer $1\le d\le r_{2k-1}+1$ with $d\in I$, say with $d\in V(T_{2s-1})$, we let $\phi^{(p)}(d)=\psi^{(p)}_{2s-1}(d)$.
			\item Let $\phi^{(p)}(r_{2k-1}+2)=\phi(r_{2k-1})=\phi(\ell_{2k}-1)$.
			\item For each integer $r_{2k-1}+3\le d\le t$, let $\phi^{(p)}(d)=\psi^{(p)}_{2k}(d-2)$ if $d-2\in V(T_{2k})$ and $\phi^{(p)}(d)=\phi(d-2)$ otherwise. 
		\end{itemize}
		We again emphasize that this map is well-defined: because $r_{2k-3}\le t-3$, we have $r_{2k-3}+2\le t$ and hence this is an element of $V(P_t)$, and similarly the last step defines a map for at least one vertex $d$.
		
		\begin{figure}
			\begin{center}
				\begin{tikzpicture}
					\draw (-4,0)--(9.5,0);
					
					\draw (4.5,0)--(5.5,0.5)--(7,0.5)--(8,0);
					\draw (4.5,0)--(5.5,1)--(7,1)--(8,0);
					\draw (4.5,0)--(5.5,1.5)--(7,1.5)--(8,0);
					
					\draw (7,0)--(8,-0.5)--(9.5,-0.5);
					\draw (7,0)--(8,-1)--(9.5,-1);
					\draw (7,0)--(8,-1.5)--(9.5,-1.5);

					\draw (0,0)--(1,0.5)--(2.5,0.5)--(3.5,0);
					\draw (0,0)--(1,1)--(2.5,1)--(3.5,0);
					\draw (0,0)--(1,1.5)--(2.5,1.5)--(3.5,0);

					\draw (-4,0.5)--(-2.5,0.5)--(-1.5,0);
					\draw (-4,1)--(-2.5,1)--(-1.5,0);
					\draw (-4,1.5)--(-2.5,1.5)--(-1.5,0);

					\draw[very thick, color=blue] (-4,1)--(-2.5,1)--(-1.5,0)--(0,0)--(1,1)--(2.5,1)--(3.5,0)--(4.5,0)--(5.5,1)--(7,1)--(8,0)--(7,0)--(8,-1)--(8.75,-1);
					
					\draw[fill=black] (8.75,-1) circle (2pt);
					
					\draw[fill=black] (8,0) circle (2pt);
					\draw[fill=black] (7,0) circle (2pt);
					\draw[fill=black] (5.5,0) circle (2pt);
					\draw[fill=black] (4.5,0) circle (2pt);
					\draw[fill=black] (3.5,0) circle (2pt);
					\draw[fill=black] (2.5,0) circle (2pt);
					\draw[fill=black] (1,0) circle (2pt);
					\draw[fill=black] (0,0) circle (2pt);
					\draw[fill=black] (-1.5,0) circle (2pt);
					\draw[fill=black] (-2.5,0) circle (2pt);
					
					\draw[fill=black] (8,-0.5) circle (2pt);
					\draw[fill=black] (8,-1) circle (2pt);
					\draw[fill=black] (8,-1.5) circle (2pt);
					
					\draw[fill=black] (7,0.5) circle (2pt);
					\draw[fill=black] (7,1) circle (2pt);
					\draw[fill=black] (7,1.5) circle (2pt);
					
					\draw[fill=black] (5.5,0.5) circle (2pt);
					\draw[fill=black] (5.5,1) circle (2pt);
					\draw[fill=black] (5.5,1.5) circle (2pt);
					
					\draw[fill=black] (2.5,0.5) circle (2pt);
					\draw[fill=black] (2.5,1) circle (2pt);
					\draw[fill=black] (2.5,1.5) circle (2pt);
					
					\draw[fill=black] (1,0.5) circle (2pt);
					\draw[fill=black] (1,1) circle (2pt);
					\draw[fill=black] (1,1.5) circle (2pt);
					
					\draw[fill=black] (-2.5,0.5) circle (2pt);
					\draw[fill=black] (-2.5,1) circle (2pt);
					\draw[fill=black] (-2.5,1.5) circle (2pt);
					
					\node at (6.25,2) {$T_{2k-1}$};
					\node at (1.75,2) {$T_{2k-3}$};
					\node at (-3.25,2) {$T_{1}$};
					\node at (5.5,-0.3) {$\ell_{2k-1}$};
					\node at (6.8,-0.4) {$r_{2k-1}$};
					\node at (8.3,0.3) {$\ell_{2k}$};
					\node at (8.75,-2) {$T_{2k}$};
					\node at (1,-0.3) {$\ell_{2k-3}$};
					\node at (2.5,-0.4) {$r_{2k-3}$};
					\node at (-2.5,-0.4) {$r_{1}$};
					\node at (-0.75,0.5) {\large{\textbf{\dots}}};
					
				\end{tikzpicture}
			\end{center}
			\caption{Finding a copy of $\mathcal{F}_{P_t,k+1}^q$ in Case 2 of \Cref{ell2 equals 2}. The graph-image of one $\phi^{(p)}$ is shown in blue.}\label{figure case 2 ell2 equals 2}
		\end{figure}
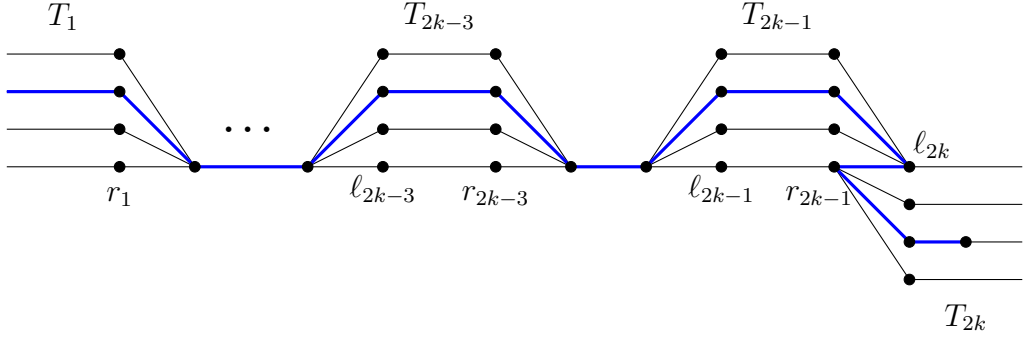
		
		Similar to the previous case, the collection $\{\phi^{(p)}\mid p\in [q]\}$ is a collection of monomorphisms of $P_t$ into $G$ such that each monomorphism agrees on the set \[R:=V(P_t)\setminus (I\cup \{r_{2k-1}+2,r_{2k}+1,r_{2k}+2,\dots,t\}),\] where we note that if $r_{2k}=t$ then $R=V(P_t)\setminus (I\cup \{r_{2k-1}+2\})$.  Then $P_t-R$ has $k+1$ connected components; $k$ of them being $T_{2s-1}$ for each $s\in [k]$ and one corresponding to some non-empty subgraph of $T_{2k}$ which in particular includes $\ell_{2k}=r_{2k-1}+1$. Thus, the union of the graph-images of the $\phi^{(p)}$'s constitutes a copy of $\mathcal{F}^q_{P_t,k+1}$ in $G$. See Figure~\ref{figure case 2 ell2 equals 2} for an illustration of this case.
	\end{proof}
	
	We can use a similar type of argument as in \Cref{ell2 equals 2} to establish the first property we need to apply \Cref{nice tuples can zigzag}.
	
	\begin{lemma}\label{nice tuples dont have gaps}
		If $G$ contains a $k$-nice $(k-1,1,C',C'',P_t)$-refinement sequence $(\Phi_t,\dots,\Phi_1)$ with $(k,t,\mathcal{T})$-nice tuple $(T_1,T_2,\dots,T_{2k})$ such that $C'\ge 2q+t$ and such that either $r_{i^*}+1\notin V(T_{i^*+1})$ for some $1\le i^*<2k$ or $\ell_{i^*}-1\notin V(T_{i^*-1})$ for some $1<i^*\le 2k$, then $G$ contains a subgraph isomorphic to an element of $\c{F}_{P_t,k+1}^q$.
	\end{lemma}
	\begin{proof}
		By the symmetry of the problem it suffices to prove the conclusion under the assumption that $r_{i^*}+1\notin [\ell_{i^*+1},r_{i^*+1}]$ for some $1\le i^*<2k$.  By \Cref{observation properties of nice tuples}, we know that $i^*$ can not be odd, so we can assume $i^*=2m$ for some integer $1\le m\le k-1$ since $i^*<2k$.  Let $\phi\in \Phi_{t}$. By the same argument as in the proof of \Cref{ell2 equals 2}, for each $T_j$ there exist maps $\psi_j^{(1)},\ldots,\psi_{j}^{(2q)}\in \Psi(T_j,\phi(N(T_j));\Phi_{t-1})$ with pairwise-disjoint images and with images disjoint from the image of $\phi$. We may further assume throughout that $\ell_2=2$ and $r_{2k-3}=t-1$, as otherwise the result holds by \Cref{ell2 equals 2}.
		
		Define $I_0:=\bigcup_{s=1}^m V(T_{2s})$ as well as the sets $I_1:=\bigcup_{s=m}^{k-1} V(T_{2s+1})$ and $I:=I_0\cup I_1$.   For each integer $1\le p\le q$, we define a map $\phi^{(p)}:V(P_t)\to V(G)$ as follows:
		\begin{itemize}
			\item We let $\phi^{(p)}(1)=\psi^{(q+p)}_{2}(2)$.
			\item For $2\le d\le t$ such that $d-1\notin I$ we let $\phi^{(p)}(d)=\phi(d-1)$.
			\item For $2\le d\le t$ such that $d-1\in I$, say with $d-1\in [\ell_j,r_j]$, we let $\phi^{(p)}(d)=\psi_j^{(p)}(d-1)$.
		\end{itemize}
		
		\begin{figure}
			\begin{center}
				\begin{tikzpicture}[scale=0.8]
					\draw (-6,0)--(12,0);
					
					\draw (-6,0)--(-5,-0.5)--(-3.5,-0.5)--(-2.5,0);
					\draw (-6,0)--(-5,-1)--(-3.5,-1)--(-2.5,0);
					\draw (-6,0)--(-5,-1.5)--(-3.5,-1.5)--(-2.5,0);
					\draw (-6,0)--(-5,-2)--(-3.5,-2)--(-2.5,0);
					\draw (-6,0)--(-5,-2.5)--(-3.5,-2.5)--(-2.5,0);
					\draw (-6,0)--(-5,-3)--(-3.5,-3)--(-2.5,0);
					
					\draw (-1,0)--(0,-0.5)--(1.5,-0.5)--(2.5,0);
					\draw (-1,0)--(0,-1)--(1.5,-1)--(2.5,0);
					\draw (-1,0)--(0,-1.5)--(1.5,-1.5)--(2.5,0);
					
					\draw (3.5,0)--(4.5,0.5)--(6,0.5)--(7,0);
					\draw (3.5,0)--(4.5,1)--(6,1)--(7,0);
					\draw (3.5,0)--(4.5,1.5)--(6,1.5)--(7,0);
					
					\draw (8.5,0)--(9.5,0.5)--(11,0.5)--(12,0);
					\draw (8.5,0)--(9.5,1)--(11,1)--(12,0);
					\draw (8.5,0)--(9.5,1.5)--(11,1.5)--(12,0);
					
					\draw[very thick, color=blue] (-5,-2.5)--(-6,0)--(-5,-1)--(-3.5,-1)--(-2.5,0)--(-1,0)--(0,-1)--(1.5,-1)--(2.5,0)--(3.5,0)--(4.5,1)--(6,1)--(7,0)--(8.5,0)--(9.5,1)--(11,1);

					\draw[fill=black] (-6,0) circle (2pt);
					\draw[fill=black] (-5,0) circle (2pt);
					\draw[fill=black] (-3.5,0) circle (2pt);
					\draw[fill=black] (-1,0) circle (2pt);
					\draw[fill=black] (-2.5,0) circle (2pt);
					\draw[fill=black] (0,0) circle (2pt);
					\draw[fill=black] (1.5,0) circle (2pt);
					\draw[fill=black] (2.5,0) circle (2pt);
					\draw[fill=black] (3.5,0) circle (2pt);
					\draw[fill=black] (4.5,0) circle (2pt);
					\draw[fill=black] (6,0) circle (2pt);
					\draw[fill=black] (7,0) circle (2pt);
					\draw[fill=black] (8.5,0) circle (2pt);
					\draw[fill=black] (9.5,0) circle (2pt);
					\draw[fill=black] (11,0) circle (2pt);
					\draw[fill=black] (12,0) circle (2pt);

					\draw[fill=black] (-5,-0.5) circle (2pt);
					\draw[fill=black] (-5,-1) circle (2pt);
					\draw[fill=black] (-5,-1.5) circle (2pt);
					\draw[fill=black] (-5,-2) circle (2pt);
					\draw[fill=black] (-5,-2.5) circle (2pt);
					\draw[fill=black] (-5,-3) circle (2pt);
					
					\draw[fill=black] (-3.5,-0.5) circle (2pt);
					\draw[fill=black] (-3.5,-1) circle (2pt);
					\draw[fill=black] (-3.5,-1.5) circle (2pt);
					\draw[fill=black] (-3.5,-2) circle (2pt);
					\draw[fill=black] (-3.5,-2.5) circle (2pt);
					\draw[fill=black] (-3.5,-3) circle (2pt);

					\draw[fill=black] (0,-0.5) circle (2pt);
					\draw[fill=black] (0,-1) circle (2pt);
					\draw[fill=black] (0,-1.5) circle (2pt);
					
					\draw[fill=black] (1.5,-0.5) circle (2pt);
					\draw[fill=black] (1.5,-1) circle (2pt);
					\draw[fill=black] (1.5,-1.5) circle (2pt);
					
					\draw[fill=black] (4.5,0.5) circle (2pt);
					\draw[fill=black] (4.5,1.5) circle (2pt);
					\draw[fill=black] (4.5,1) circle (2pt);
					
					\draw[fill=black] (6,0.5) circle (2pt);
					\draw[fill=black] (6,1) circle (2pt);
					\draw[fill=black] (6,1.5) circle (2pt);
					
					\draw[fill=black] (9.5,0.5) circle (2pt);
					\draw[fill=black] (9.5,1) circle (2pt);
					\draw[fill=black] (9.5,1.5) circle (2pt);
					
					\draw[fill=black] (11,0.5) circle (2pt);
					\draw[fill=black] (11,1) circle (2pt);
					\draw[fill=black] (11,1.5) circle (2pt);

					\node at (0.75,1) {$T_{2s}$};
					\node at (5.25,-1) {$T_{2s+1}$};
					\node at (10.25,-1) {$T_{2k-1}$};
					\node at (-4.25,1) {$T_{2}$};
					\node at (-1.75,-0.5) {\large{\textbf{\dots}}};
					\node at (7.75,0.5) {\large{\textbf{\dots}}};
					
					\node at (-5,0.3) {{\footnotesize$\phi(\ell_{2})$}};
					\node at (-3.5,0.3) {{\footnotesize$\phi(r_{2})$}};
					\node at (0,0.3) {{\footnotesize$\phi(\ell_{2s})$}};
					\node at (1.5,0.3) {{\footnotesize$\phi(r_{2s})$}};
					\node at (4.4,-0.3) {{\footnotesize$\phi(\ell_{2s+1})$}};
					\node at (6.1,-0.3) {{\footnotesize$\phi(r_{2s+1})$}};
					\node at (9.4,-0.3) {{\footnotesize$\phi(\ell_{2k-1})$}};
					\node at (11.1,-0.3) {{\footnotesize$\phi(r_{2k-1})$}};
				\end{tikzpicture}
			\end{center}
			\caption{Finding a copy of $\mathcal{F}_{P_t,k+1}^q$ in \Cref{nice tuples dont have gaps}. The graph-image of one $\phi^{(p)}$ is shown in blue.}\label{figure if there is a gap we win}
		\end{figure}
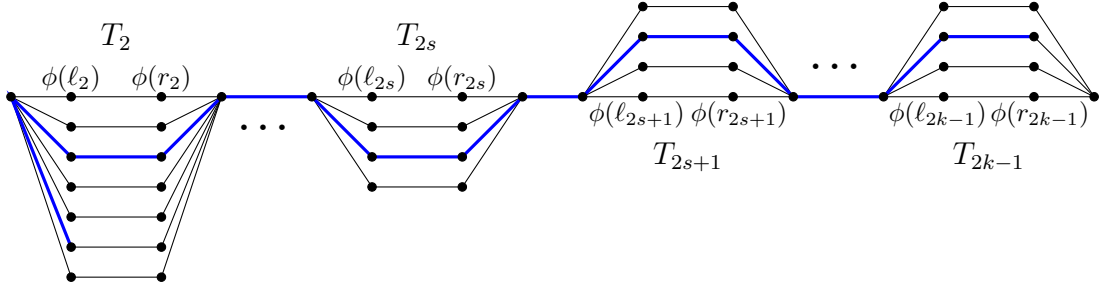
		
		Since $\Phi_{t-1}$ is distinguishing and by our choices of the $\psi^{(p)}_j$'s, the collection $\{\phi^{(p)}\mid p\in [q]\}$ is a collection of monomorphisms of $P_t$ into $G$ such that each monomorphism agrees on the image of $R:=V(T)\setminus (I'\cup \{1\})$ where $I':=\{d\mid d-1\in I\}$, and otherwise the images are pairwise disjoint. With this, $P_t-R$ has $k+1$ components; one given by the singleton $\{2\}$, $m$ given by $T_{2s}$ for $s\in [m]$, and $k-m$ given by $T_{2s+1})$ for $s\in [k-1]\setminus [m-1]$. Taken together the graph-images of the $\phi^{(p)}$'s yield a copy of $\mathcal{F}^q_{P_t,k+1}$ in $G$. See Figure~\ref{figure if there is a gap we win} for an illustration.
	\end{proof}
	
	Up to this point all of our results hold for arbitrary $k,t$, but as mentioned around the statement of \Cref{k-2 edges for paths}, we know that our conclusion can only hold if $t\not\equiv 2 \mod k-1$. The following observation, which involves finding a ``small'' tree $T_{i^*}$ to use in \Cref{nice tuples can zigzag}, is the only place where we use this condition.
	
	\begin{lemma}\label{nice tuples have a small tree}
		If $(T_1,T_2,\dots,T_{2k})$ is a $(k,t,\mathcal{T})$-nice tuple for some collection $\mathcal{T}$ of subtrees of $P_t$ and if $t\not \equiv 2\mod k-1$,  then there exists some $1<i^*<2k$ such that 
		\[
		|V(T_{i^*})|\leq  \frac{t-3}{k-1}-1.
		\]
	\end{lemma}
	
	\begin{proof}
		For each even integer $2\le i\le 2k-2$ let $I_i:=[\ell_i+1,r_{i+1}-1]$, where here we take $I_i:=\emptyset$ whenever $r_{i+1}-1<\ell_i+1$.  
		
		We claim that each $I_i$ set is disjoint from $S:=\{r_1,\ell_2,r_3,\ell_4,\ldots,r_{2k-1},\ell_{2k}\}$. Indeed, fix some $j\in [2k]$ and some even $i\in [2k]$. Using \Cref{observation properties of nice tuples}, if $j$ is odd and $i<j$, then $r_j\not\in I_i$ since $r_j\geq r_{i+1}>r_{i+1}-1$. If $j$ is odd and $i>j$, then  $r_j\not\in I_i$ since $r_j<\ell_{j+1}\leq \ell_i<\ell_i+1$. Similarly, if $j$ is even and $i<j$, then $\ell_j\not\in I_i$ since $\ell_j>r_{j-1}\geq r_{i+1}>r_{i+1}-1$. Finally if $j$ is even and $i\geq j$, then $\ell_j\not\in I_i$ since $\ell_j\leq \ell_i<\ell_i+1$. Thus, $I_i$ is indeed disjoint from $S$.
		
		We also claim that $I_i\cap I_j=\emptyset$ for all $i\ne j$. Indeed if $i,j$ are even integers with $i<j$, then this follows from the fact that $r_{i+1}-1\leq r_{j-1}-1< \ell_j-1<\ell_j+1$.
		
		With these two claims in mind, the pigeonhole principle implies that there exists some even $i^*$ with $2\le i^*\le 2k-2$ such that
		\[
		|I_{i^*}|\leq \left\lfloor \frac{t-2k}{k-1} \right\rfloor=\left\lfloor \frac{t-2}{k-1}\right\rfloor-2=\left\lfloor \frac{t-3}{k-1}\right\rfloor-2,
		\]
		where this last step used $t\not \equiv 2\mod k-1$.  
		
		Now, $V(T_{i^*})=[\ell_{i^*},r_{i^*}]\sub \{\ell_{i^*}\}\cup I_{i^*}$ since $r_{i^*}\leq r_{i^*+1}-1$ by \Cref{observation properties of nice tuples}. It follows then that 
		\[
		|V(T_{i^*})|\leq 1+|I_{i^*}|\leq \left\lfloor \frac{t-3}{k-2}\right\rfloor-1\leq \frac{t-3}{k-2}-1
		\]
		as desired.
	\end{proof}
	
	\subsection{Proof of \texorpdfstring{\Cref{path reduce by 1}}{Theorem 5.1}}
	We now put all of these pieces together to prove the result.
	
	\begin{proof}[Proof of \Cref{path reduce by 1}]
		Fix $q\in\mathbb{N}$ and assume $G$ is $\mathcal{F}^q_{P_t,k+1}$-free and that it is not the case that the number of copies of $P_t$ in $G$ is $O(v(G)^{k-1}e(G))$. Let $C_0$ denote the constant guaranteed by \Cref{lemma few copies implies k-nice refinement sequence}, and let $C':=\max\left\{C_0,1+q+(t-1)(1+q(\frac{t-3}{k-1}-1)),2q+t\right\}$. By \Cref{lemma few copies implies k-nice refinement sequence}, $G$ contains a $k$-nice $(k-1,1,C',C'',P_t)$-refinement sequence (where $C''$ is the constant guaranteed by \Cref{good sequence or few copies}), say with associated subtree system $\mathcal{T}$ and associated $(k,t,\mathcal{T})$-nice tuple $(T_1,T_2,\dots,T_{2k})$.
		
		By \Cref{nice tuples have a small tree}, there exists some $i^*$ with $1<i^*<2k$ such that $|V(T_{i^*})|\leq  \frac{t-3}{k-1}-1$. By \Cref{nice tuples dont have gaps}, we may assume that $r_{i^*}+1\in V(T_{i^*+1})$ and $\ell_{i^*}-1\in V(T_{i^*-1})$. But then \Cref{nice tuples can zigzag} gives us that $G$ is not $\mathcal{F}^q_{P_t,k+1}$-free, completing the proof.
	\end{proof}

	\section{Concluding Remarks}\label{section concluding remarks}
	In this paper, we proved that for any tree $T$, integer $k$, and family of graphs $\c{F}$ that either $\ex(n,T,\c{F})=\Omega(n^{k+1})$ or $\ex(n,T,\c{F})=O(\ex(n,\c{F})^k)$, with this upper bound being tight for $T=P_t$ with $t\equiv 2\mod k-1$.  Moreover, when $T=P_t$ with $t\not \equiv 2\mod k-1$ we refined this upper bound to $\ex(n,P_t,\c{F})=O(\ex(n,\c{F})^{k-1}n)$ which is tight whenever $t\equiv 3\mod k-1$.  In general we believe the following refinement for paths should be true.
	
	\begin{conjecture}\label{optimal path conjecture}
		Let $k,t$ be integers with $k\ge 2$ and $t\ge 2k+1$.  If $0\le r\le k-2$ is the unique integer such that $t-3\equiv r \mod k-1$, then every family of graphs $\c{F}$ either satisfies
		\[\ex(n,P_t,\c{F})=O(\ex(n,\c{F})^{r+2}n^{k-2-r}),\]
		or $\ex(n,P_t,\c{F})=\Omega(n^{k+1})$.
	\end{conjecture}
	We note that these bounds would be best possible if true (see Appendix~\ref{section appendix} for more), and that our main results imply this conjecture in the cases $r=k-2,k-3$.
	
	There are a number of obstacles towards extending our ideas to prove \Cref{optimal path conjecture}.  For example, we lack an effective Helly theorem for piercing subtrees with a given number of vertices and edges.
	
	\begin{problem}\label{mixed Helly problem}
		For all integers $a,b\ge 0$, determine the smallest number $h=h(a,b)$ such that the following holds: if $P_t$ is a tree and $\c{T}$ is a subtree system such that for every $\c{T}'\sub \c{T}$ of size at most $h$ there exists a set of at most $a$ edges and a set of at most $b$ vertices which pierces $\c{T}'$, then there exists a set of at most $a$ edges and at most $b$ vertices which pierces $\c{T}$.
	\end{problem}
	
	We emphasize that we did not directly use a result of this form for proving our path refinement \Cref{k-2 edges for paths}, but such a result with $a=k-1$ and $b=1$ served as our initial motivation for how to go about the proof, and it seems such a result in general may be useful for finding a proof for \Cref{optimal path conjecture}.  
	
	\bibliographystyle{amsplain}
	\bibliography{bib}{}
	
	\appendix
	
	\section{Tightness for Paths}\label{section appendix}
	In this appendix we verify that Theorems~\ref{k-1 edges for trees} and \ref{k-2 edges for paths} are tight for paths of certain lengths.  More generally, we show the following in the spirit of \Cref{optimal path conjecture}.
	\begin{proposition}\label{tightness for paths}
		Let $k,t$ be integers with $k\ge 2$ and $t\ge 2k+1$.  If $0\le r\le k-2$ is the unique integer such that $t-3\equiv r \mod k-1$, then there exists $q$ sufficiently large so that
		\[\ex(n,P_t,\c{F}_{P_t,k+1}^q)=\Omega(\ex(n,\c{F}_{P_t,k+1}^q)^{r+2}n^{k-2-r}).\]
	\end{proposition}
	
	To prove this we will need a technical tool based on the breakthrough result of Bukh and Conlon~\cite{BC2018} using random polynomial graphs to solve a finite family version of the rational exponents conjecture.
	
	To this end, we say a pair $(F,R)$ is a \textbf{rooted graph} if $F$ is a graph and $R\sub V(F)$.  Given a graph $F$ and a set of vertices $S\sub V(F)$, we define $e_S(F)$ to be the number of edges of $F$ which are incident to a vertex of $S$.  We define the \textbf{rooted density} of a rooted graph $(F,R)$ to be $\min_{S\sub V(F)\sm R} \frac{e_S(F)}{|S|}$.
	
	\begin{theorem}[\cite{S24}, Proposition 2.1]\label{random polynomials general}
		If $H$ is a graph and if $(F_1,R_1),\ldots,(F_t,R_t)$ are rooted graphs with rooted density at least $b/a$ for some integers $a,b\ge 1$, then there exists some $q_0$ depending only on $H$ and $\{(F_1,R_1),\ldots,(F_t,R_t)\}$ such that for all $q\ge q_0$,
		\[
		\ex(n,H,\{(F_1)_{R_1}^q,\ldots,(F_t)_{R_t}^q\})=\Omega(n^{v(H)-\frac{a}{b}e(H)}).
		\]
	\end{theorem}
	
	Actually, the result from~\cite{S24} is slightly stronger and allows one to forbid not just $(F_i)_{R_i}^q$ (which is the graph obtained by taking $q$ copies of $F_i$ which agree on $R_i$ and which are otherwise disjoint), but in fact the family of graphs which can be obtained by taking $q$ copies of $F_i$ which agree on $R_i$ (but which are not necessarily disjoint outside of $R_i$).
	
	One can check that the rooted path graph $(P_{b+1},\{v_1,v_{b+1}\})$ has rooted density $\frac{b}{b-1}$, and also that $(P_{b+1})_{\{v_1,v_{b+1}\}}^q$ is the theta graph $\theta_{q,b}$.  As such, we obtain the following corollary of \Cref{random polynomials general}.
	
	\begin{corollary}\label{random polynomials theta}
		For every graph $H$ and integer $b\ge 2$, there exists some $q$ sufficiently large such that
		\[\ex(n,H,\{\theta_{q,2},\theta_{q,3},\ldots,\theta_{q,b}\})=\Omega(n^{v(H)-(1-1/b)e(H)}).\]
	\end{corollary}
	We now prove \Cref{tightness for paths}.

	\begin{proof}[Proof of \Cref{tightness for paths}]
		Let $b=\lfloor \frac{t-3}{k-1}\rfloor$, noting by hypothesis $b\ge 2$ and $b=\frac{t-3-r}{k-1}$.  We aim to show that for $q$ large, \[\ex(n,\c{F}_{P_t,k+1}^q)=\Theta(n^{1+1/b})=\Theta(n^{\frac{k+t-4-r}{t-3-r}}),\] and \[\ex(n,P_t,\c{F}_{P_t,k+1}^q)=\Omega(n^{\frac{(k+t-4-r)(r+2)}{t-3-r}+k-2-r}),\] from which the result follows.
		
		The fact that $\ex(n,\c{F}_{P_t,k+1}^q)=O(n^{1+1/b})$ follows from the same proof as \Cref{family bound given diameter} where we in fact showed exactly the bound $O(n^{1+1/b})$ in \eqref{eq:diameter bound} before replacing it with $O(n^{1+\frac{k-1}{d-k}})$ in order to state a simpler which bound which remained valid for $k=1$.
		
		For both lower bounds, we claim that every element of $\c{F}_{P_t,k+1}^q$ contains a theta graph $\theta_{q,b'}$ for some $2\le b'\le b$ as a subgraph.  Indeed, say $(P_t)_R^q\in \c{F}_{P_t,k+1}^q$, meaning that $R$ is a set of vertices such that $P_t-R$ has at least $k+1$ connected components.  It is not difficult to see that this means $|R|\ge k$ and that $P_t-R$ can be written as the union of subpaths $P^{(1)},\ldots,P^{(s)}$ for some $s\ge k+1$.  At most two of these subpaths contain $1$ or $t$, and among the remaining $s-2\ge k-1$ subpaths, the pigeonhole principles implies there must exist some $i$ with
		\[|V(P^{(i)})|\le \frac{|V(P_t)\sm (R\cup \{1,t\})|}{k-1}\le \frac{t-k-2}{k-1}=\frac{t-3}{k-1}-1,\]
		and since $|V(P^{(i)})|$ is an integer this means $|V(P^{(i)})|\le b-1$.  Write the vertex set of this subpath as $\{\ell+1,\ell+2,\ldots,\ell+b'-1\}$ for some $\ell,b'$, noting that $\ell+1\ne 1$ and $\ell+b'-1\ne t$ by hypothesis.  The fact that $P^{(i)}$ is a component of $P_t-R$ then means $\ell,\ell+b'\in R$, and this means the graph $(P_t)_R^q$ contains $\theta_{q,b'}$ as a subgraph, namely the subgraph obtained by taking the $q$ copies of $P^{(i)}$ which all agree on $\ell$ and $\ell +b'$.  Moreover, we have $b'=|V(P^{(i)})|+1\le b$, proving the claim.
		
		With this claim in mind, we can apply \Cref{random polynomials theta} with $H=K_2$  to obtain
		\[\ex(n,\c{F}_{P_t,k+1}^q)\ge \ex(n,\{\theta_{q,2},\theta_{q,3},\ldots,\theta_{q,b}\})=\Omega(n^{1+1/b}).\]
		Similarly, \Cref{random polynomials theta} with $H=P_t$ together with the claim gives
		\begin{equation}\label{equation path tightness}
			\ex(n,P_t,\c{F}_{P_t,k+1}^q)\ge\ex(n,P_t,\{\theta_{q,2},\theta_{q,3},\ldots,\theta_{q,b}\})=\Omega(n^{t-(1-1/b)(t-1)}).
		\end{equation}
		After some standard algebraic manipulations, using the fact that $b=\frac{t-3-r}{k-1}$,~\eqref{equation path tightness} is equivalent to
		\[
		\ex(n,P_t,\c{F}_{P_t,k+1}^q)=\Omega(n^{\frac{(k+t-4-r)(r+2)}{t-3-r}+k-2-r}),
		\]
		as desired.
	\end{proof}

\end{document}